\documentclass[10pt]{article}
\usepackage[top=4cm,left=2.8cm,right=2.8cm,bottom=3.5cm]{geometry}
\usepackage{amsmath,amsthm}
\usepackage{amsfonts}
\usepackage{amssymb}
\usepackage{graphicx}
\usepackage{rotating}
\usepackage{psfrag}
\usepackage{color}
\usepackage{soul}
\usepackage{cite}
\usepackage{tikz}
\usetikzlibrary{er,positioning}
\usepackage{graphics}
\usepackage{graphicx}
\usepackage{algpseudocode,algorithm,algorithmicx}
\usepackage{xspace}
\usepackage{adjustbox}
\usepackage{subcaption}
\usepackage{setspace}
\usepackage{verbatim}
\usepackage{fullwidth}
\usepackage{float}
\usepackage{listings}
\usepackage{lipsum}
\usepackage{framed}
\usepackage{caption}
\usepackage{booktabs,multirow,array}
\usepackage{multirow}
\usepackage{multicol}
\usepackage{mathtools}
\mathtoolsset{showonlyrefs,showmanualtags}
\usepackage{hyperref}
 \usepackage{array}
\usepackage{colortbl}
\usepackage[title]{appendix}
\usepackage{etoolbox}
\usepackage[font=small,labelfont=bf]{caption}
\usepackage{amsmath,amssymb}
\usepackage{amsthm}
\usepackage{enumitem}
\usepackage{graphicx}
\usepackage{epstopdf }
\usepackage{setspace}
\usepackage[english]{babel}
\usepackage{float}
\usepackage{chngcntr}
\usepackage{cite}
\usepackage{physics}
\usepackage{authblk}
\usepackage{hyperref}
\usepackage{subcaption}
\counterwithout{equation}{section}
\usepackage[retainorgcmds]{IEEEtrantools}
\usepackage{mathptmx}
\usepackage{mathrsfs}

\newcommand{\ep}{\epsilon}

\newcommand{\cK}{\mathcal{K}}
\newcommand{\cP}{\mathscr{P}}

\definecolor{addc}{rgb}{0.0,0.50,0.0}

\theoremstyle{plain}
\numberwithin{equation}{section}
\newtheorem{lemma}{Lemma}[section]      
\newtheorem{sublemma}{Lemma}[lemma]     
\newtheorem{thm}{Theorem}[section]

\theoremstyle{definition}
\newtheorem{definition}{Definition}[section]

\theoremstyle{remark}
\newtheorem{myrem}{Remark}[section]
\numberwithin{equation}{section}
\begin{document}

\title{\Large 
\emph{A Posteriori} Error Analysis of Finite Element Approximation to an Evolution Equation with a Positive-Type Memory Term}

\author{Sandhya Mel$^{1}$, Amiya Kumar Pani$^{2}$\footnote{Corresponding author:amiyap@goa.bits-pilani.ac.in}, Gujji Murali Mohan Reddy$^{1}$
\\{\small $^{1}$Department of Mathematics,}\\{\small BITS-Pilani Hyderabad Campus,
Telangana - 500078, India}
\\{\small $^{2}$Department of Mathematics,}\\{\small BITS-Pilani, KK Birla Goa Campus, NH 17 B, Zuarinagar,}\\{\small Goa - 403726, India}}
\date{}
\maketitle
\begin{abstract}
This paper develops reliable \emph{a posteriori} error estimates for finite element approximation of an evolution equation with a positive-type memory kernels on bounded convex polygonal or polyhedral domains, including both smooth and weakly singular kernels. For the weakly singular kernel $t^{\alpha-1}/\Gamma(\alpha),\; 0<\alpha<1,$ the model captures the transition from parabolic diffusion to wave propagation, recovering the corresponding limiting equations as $\alpha\to0^+$ and $\alpha\to1^-$, respectively, with the memory term governing the intermediate dynamics.
The spatial discretization employs time-dependent finite element spaces, while the backward Euler method combined with numerical quadrature is used for time integration. A unified \emph{a posteriori} analysis is established for both semidiscrete and fully discrete approximations. Elliptic reconstruction yields optimal spatial error estimates, while a memory-lag estimate preserves the essential cancellation in the convolution term, yielding a first-order temporal estimator whose kernel dependence is characterized by its integral. The resulting estimators are robust, with constants independent of the final time, spatial mesh size, and time step. Their kernel dependence involves only computable quantities, avoiding exponential growth in time. Moreover, for the weakly singular kernel, the $\alpha$-dependent factors remain uniformly bounded for $\alpha\in(0,1).$ Numerical experiments on a convex domain validate the theoretical results, and a space–time adaptive algorithm is also demonstrated on an L-shaped domain featuring a corner singularity and two temporal layers.
\end{abstract}
\section{Introduction}\label{sec:1}

This paper deals with \emph{a posteriori} error estimators for the following  evolution equation with a positive-type kernel covering both smooth and weakly singular kernels.
Now, find $u(x,t)$ defined on the space-time cylinder $\Omega \times [0,T]$ such that 
\begin{align}
 u_{t}(x,t)+ \int_{0}^t \mathcal{K}(t-s) Au(x,s)ds &= f(x,t), \;\;\;\;\; (x,t) \in \Omega \times (0,T], \label{ee}\\
 u(x,0) &= u_{0}(x), \;\;\;\; x \in \Omega, \nonumber \\
 u(x,t) &= 0,\;\;\;\;\;\;\;\;\;\; (x,t) \in \partial \Omega \times (0,T], \nonumber 
\end{align}
where $\Omega$ is a bounded convex polygonal or polyhedral domain in  $\mathbb{R}^d (d = 2,3)$ with boundary $\partial \Omega$. 
Here,  $u_{t}(x,t)=\frac{\partial }{\partial t}u(x,t)$ and $A$ is a self-adjoint, positive definite second-order linear  elliptic partial differential operator, say  of the form:
\begin{equation*}
    Au=-\nabla\cdot(\mathbf{A}\nabla u),
\end{equation*}
where $``\nabla"$ denotes the  spatial gradient and the matrix  
$\mathbf{A}=\{a_{ij}(x)\}$ is a $d\times d$ positive definite matrix  assumed to be in 
$L^{\infty}(\Omega)^{d\times d}$ in space variables. 
The real-valued functions $u_0(x)$ and $f(x,t)$ are assumed to be smooth. Further, we assume that the kernel $\mathcal{K}$ is locally integrable on $[0,\infty)$ and is 
positive definite, i.e.
\begin{equation} \label{PK}
     \int^T_{0} \phi(t)\left(\int_{0}^t \mathcal{K}(t-\tau)\phi(\tau) d\tau \right) dt\ge 0\;\;\;\;\;\;\;\;  \forall T>0, \;\;\;\phi\in C([0,T]).
\end{equation}
For more details, see \cite{MT93,LS96}.

This paper examines two different classes of positive definite kernels: (i) when 
$\mathcal{K}(t)$ is smooth and (ii) when $\mathcal{K}(t)$ has a weak singularity at 
$t=0$. 

For smooth kernels, we assume that there exist constants $\gamma,\gamma_1 > 0$ such 
that, for all $t,s \in [0,T]$, with $s\leq t$, 
 \begin{align} \label{kernel}
     |\mathcal{K}(t-s)|\leq \gamma,\;\;\;\mbox{and}\;\; 
   |\mathcal{K}'(t-s)|\leq \gamma_1, 
 \end{align}
where $\mathcal{K}'$ denotes the derivative of $\mathcal{K}$ with respect to its argument,
whereas for weakly singular kernel, we assume that $\mathcal{K} \in L^1(0,T)$ and takes 
the form
\begin{equation}
\mathcal{K}(t-s)=\frac{(t-s)^{\alpha-1}}{\Gamma(\alpha)}, \label{wsk}
 \end{equation}
where $0<\alpha<1$ and $\Gamma(\cdot)$ denotes Euler's gamma function. The model exhibits a continuous transition from diffusion-dominated to wave-like dynamics as the fractional parameter $\alpha$ varies from $0$ to $1.$ Specifically, it recovers the classical parabolic equation in the limit $\alpha\to0^+$ and a wave-type equation as $\alpha\to1^-.$ Consequently, for $\alpha\in(0,1),$ the model captures the intermediate dynamical regime, in which the nonlocal memory term reflects characteristics of both diffusive and wave-like behavior.

Throughout, we use the primitive
\begin{equation}\label{K1}
\mathcal{K}_1(t):=\int_0^t |\mathcal{K}(s)|\,ds,
\qquad
\left\{
\begin{aligned}
\mathcal{K}_1(t)&\leq \gamma t
&&\text{under \eqref{kernel}},\\[4pt]
\mathcal{K}_1(t)&=\frac{t^{\alpha}}{\Gamma(\alpha+1)}
&&\text{for \eqref{wsk}}.
\end{aligned}
\right.
\end{equation}
The positive definiteness of $\mathcal{K}$ in conjunction with a standard energy 
argument  easily leads to stability results, see  \cite{MT93}. 

Evolution equations with a memory term containing positive-type kernels hold 
particular significance due to their prevalence in various physical phenomena. 
These equations are commonly encountered in diverse applications, such as modelling heat conduction in materials with memory \cite{M77} and describing compressible viscoelastic media \cite{HGS70}. While the inclusion of the Volterra integral term enhances the accuracy of the description of physical phenomena, it introduces novel 
challenges in both theoretical frameworks (owing to limited regularity \cite{RS15})
and numerical computations (including substantial data storage \cite{MTW96}). 
These challenges set these equations apart from purely parabolic or hyperbolic 
problems. 

There is a significant amount of literature on \emph{apriori} error estimates 
for equations of the form \eqref{ee}, see \cite{MT93, MTW96, MT04, MT10, CP03, LS96, CLP06, YF88}. 
For the time discretization of \eqref{ee}, \emph{apriori} and \emph{a posteriori} error estimates for a discontinuous Galerkin scheme of backward Euler (BE) type, incorporating sparse quadrature in an abstract Hilbert space setting, are established in \cite{AEL03} and subsequently are employed to develop an adaptive time stepping strategy. The \emph{a posteriori} analysis yields error bounds at the discrete time levels through a duality argument that relies on stability estimates for the adjoint problem. However, the analysis is limited to temporal discretization and does not account for spatial discretization errors. Moreover, the resulting error bound contains a logarithmic factor that depends on the ratio of the final time to the smallest time step, which may affect the sharpness of the estimate.
\emph{A posteriori} error estimates are also available for the related time-fractional subdiffusion problems
\cite{BM22,K22}, whose structure differs from that of \eqref{ee}. To the best of our knowledge, however,
\emph{a posteriori} error estimates for fully discrete finite element approximations of \eqref{ee} that account
for both the spatial and the temporal discretization errors have not been derived so far. This article attempts to address this gap. We now present relevant literature on techniques developed for time-dependent problems. \emph{A posteriori} error estimation of time-dependent problems is investigated by several researchers \cite{EJ91,P98, BBM05, 
FGN07,V03, GLM13} and the references therein. In particular, the elliptic 
reconstruction operator is introduced in \cite{MN03} to restore the optimal order in 
space using energy techniques. Subsequently, several other problems have been
investigated \cite{LM06,AMN06,LMP15,DLM09,BKM12} using the reconstruction technique to 
establish the optimal rate for \emph{a posteriori} convergence analysis. More 
recently, reliable \emph{a posteriori} error estimates have been established for 
parabolic integro-differential equations with smooth kernels using the Ritz-Volterra reconstruction operator \cite{RS15, RSC19, R22}, a generalization of the 
elliptic reconstruction operator.

While elliptic reconstruction is well-established for standard parabolic problems, its application to \eqref{ee} is obstructed by the nonlocal memory kernel, which couples the spatial reconstruction at time $t$ to the entire solution history. This coupling introduces several novelties in this work.

\noindent\textbf{Main contributions.} The analysis builds upon the elliptic reconstruction technique \cite{MN03,LM06}. The main contributions of this article are as follows.
\begin{itemize}
\item \emph{Semidiscrete estimate.} For the spatially semidiscrete Galerkin approximation, we derive a computable upper residual bound for the error in the $L^\infty(0,T; L^2(\Omega))$-norm (Theorem~\ref{Thm:3.1}). The estimate is obtained without any Gronwall argument, and the constants that appear in the bounds are independent of the mesh parameter, $T$, and the kernel.

\item \emph{Fully discrete estimate on time-varying meshes.} For the BE scheme combined with a positive-type quadrature rule and finite element spaces that may change from one time level to the next, we prove an upper bound in the same norm (Theorem~\ref{Final_Th}) in terms of fully computable estimators. Each estimator is attributed to one source of error: elliptic reconstruction ($\eta^{\mathrm{rec}}_n$), space ($\eta^S_n$), time ($\eta^T_n$), quadrature ($\eta^{Q}_n$), memory lag and history mesh change ($\eta^{K}_n=\eta^L_n+\eta^H_n$), and mesh change ($\eta^M_n$). 

\item \emph{A memory-lag estimate.} 
For the fully discrete scheme, the memory term is handled using a memory-lag estimate (Lemma~\ref{lem:4.6.3}) that retains the cancellation between the memory at the current time and at the discrete time level; the resulting estimator is first-order in time, consistent with the BE scheme, and depends on the kernel only through its integral.
The history mesh-change estimator $\eta^H_n$ (Lemma~\ref{lem:4.6.4}) measures the effect of mesh modification on the discrete history and vanishes when the mesh is not changed. Together with $\eta^M_n$, the memory-lag estimate quantifies the effect of mesh changes satisfying the
compatibility and relative-refinement assumptions of Section \ref{sec:4} by computable quantities.
In contrast to \cite{AEL03}, we derive an \emph{a posteriori} error estimate in the $L^\infty(0,T;L^2(\Omega))$-norm that holds uniformly for all $t\in[0,T],$ rather than only at the discrete time levels. The analysis in this article is based on an energy argument that relies solely on the positive definiteness property \eqref{PK} of the kernel. Furthermore, the resulting error bound does not involve a logarithmic factor that depends on the ratio of the final time to the smallest time step size. By combining elliptic reconstruction to control the spatial discretization error with a memory-lag estimate to tackle the temporal nonlocality, we obtain a unified error analysis for the fully discrete approximation.

\item \emph{Unified treatment of the kernels without Gronwall's lemma.} Both smooth and weakly singular kernels are treated within one framework. The analysis uses the positive definiteness \eqref{PK} of the kernel instead of Gronwall's lemma, so that the bounds do not grow exponentially in time. The bound is robust in the sense that the constants appearing in the estimate are independent of the final time, time step sizes, and mesh sizes; the kernel enters only through the computable quantities. In particular, for $\cK(t)=t^{\alpha-1}/\Gamma(\alpha)$, the $\alpha$-dependent factors in the estimators remain uniformly bounded for all $\alpha\in(0,1)$, including in the limiting regimes $\alpha\to0^+$ and $\alpha\to1^-$.
\item \emph{Numerical validation.} Numerical experiments for a smooth and a weakly singular kernel confirm the theoretical findings. In addition, for adaptivity, we generate a space-time singular benchmark solution with a corner singularity in space and two layers in time by choosing a classical 
L-shaped-domain benchmark and present a space-time adaptive algorithm based on estimators. The computations use FreeFEM and MATLAB.  The algorithm refines the mesh near the corner and where the data vary, grades the time steps at the two temporal layers, and reduces the error at a rate close to the inverse square root of the total number of unknowns. 
\end{itemize}

The paper is organized as follows. In Section \ref{sec:2}, we introduce the notation and the variational formulation of \eqref{ee}.
In Section \ref{sec:3}, \emph{a posteriori} error estimates are derived for the 
spatially semidiscrete scheme. \emph{A posteriori} error estimates are established 
for the fully discrete BE scheme in Section \ref{sec:4}. In Section \ref{sec:5}, 
numerical investigations are performed. Finally, we conclude the article in Section \ref{sec:6}.

\section{Notation and preliminaries}\label{sec:2}
For a Lebesgue measurable open set $\mathcal{X} \subset \Omega$, and for $1\leq p \leq \infty$, we define the standard  $ L^{p}(\mathcal{X})$ spaces with norm $\|\cdot\|_{L^p(\mathcal{X})}.$

For $p = 2$, $L^2(\mathcal{X})$ is a Hilbert space with the inner-product
  \begin{equation*}
   \langle \phi, \psi \rangle_{\mathcal{X}} = \int_{\mathcal{X}} \phi \psi d\mathbf{x},~~ \phi, \psi \in L^2(\mathcal{X}),
 \end{equation*} 
that induces the norm $\lVert \cdot \rVert_{L^2(\mathcal{X})}$. When $\mathcal{X} = \Omega$, we omit the subscripts on the norm and the inner product. 

For $m \in \mathbb{N}$ and $p=2$, the Sobolev space $H^{m}(\mathcal{X})$ is a Hilbert space of order $m$ on $\mathcal{X}$ defined as
\[
H^{m}(\mathcal{X})=\{\psi\in L^{2}(\mathcal{X}):~ {\partial}^{\Upsilon} \psi\in L^{2}(\mathcal{X})~\text{for~all~multi-indices}~\Upsilon, ~0 \leq|\Upsilon| \leq m \},
\]
with the inner-product
\[
\langle\phi, \psi\rangle_{H^m(\mathcal{X})} = \underset{0 \leq|\Upsilon| \leq m}{\sum}
\int_{\mathcal{X}} \partial^{\Upsilon}\phi\partial^{\Upsilon} \psi d\mathbf{x},~~
\phi, \psi\in H^{m}(\mathcal{X}),
\]
and the induced norm 
\begin{eqnarray}
  \|\psi\|_{H^m(\mathcal{X})} 
 =\left(\underset{0 \leq |\Upsilon| \leq m}{\sum} \|\partial^{\Upsilon} \psi\|_{L^2(\mathcal{X})}^2  \right)^{\frac{1}{2}} \nonumber,
\end{eqnarray}
where $\partial^{\Upsilon}$ denotes the partial derivative corresponding to the multi-index $\Upsilon$.
In particular, when $m = 1$ or $2$ and $\mathcal{X} = \Omega$, we use the notation $\|\psi\|_m$ to represent the norm on $H^m(\mathcal{X})$ and $|\psi|_{m}$ to represent the corresponding seminorm. 
On a subset $\mathcal{X}\subset\Omega$ we write $|\psi|_{H^m(\mathcal{X})}$ for the seminorm and $\|\cdot\|_{\mathcal{X}}:=\|\cdot\|_{L^2(\mathcal{X})}$; thus $|\psi|_m=|\psi|_{H^m(\Omega)}$.
The space $H^1_0(\Omega)$ signifies the space of functions in $H^1(\Omega)$ that vanish on the boundary of $\Omega$ in the sense of trace.

Introduce the Bochner space \( L^2(0, T; L^2(\Omega))\) as 
the space of functions that are \( L^2(\Omega) \)-valued and square-integrable over the time interval \( [0, T] \). Similarly, we define the spaces 
\(L^{\infty}(0, T; L^2(\Omega))\) and \(L^{\infty}(0, T; H_0^{1}(\Omega))\).

Given $f \in L^{2}(0,T; L^{2}(\Omega))$ and $u_{0} \in L^2(\Omega)$, the variational formulation of the problem \eqref{ee} is to seek  $u \in L^{\infty}(0, T; H_0^{1}(\Omega))$ with $u_{t} \in L^{2}(0,T; L^{2}(\Omega))$ such that
\begin{align}\label{WF}
\langle u_{t},\phi \rangle + \int_{0}^{t} \mathcal{K}(t-s) a(u(s),\phi) ds&=\langle f,\phi \rangle ~~~~\forall \phi\in H_{0}^{1}(\Omega),~ a.e. ~ t \in [0,T],  \\
u(0) &=u_{0}, \nonumber
\end{align}
where $a(\cdot,\cdot)$ is a bilinear form on $H_{0}^{1}(\Omega)$ defined by
\begin{equation}
a(v,\phi)=\langle \mathbf{A} \nabla v,\nabla \phi \rangle, ~~~~~ \forall  v,\phi\in H_{0}^{1}(\Omega).
\end{equation}
We further assume that the bilinear form $a(\cdot,\cdot)$ is continuous and coercive on $H_{0}^{1}(\Omega)$, i.e., 
\begin{align}
    a(\psi,\phi)& \leq\beta |\psi|_1 |\phi|_1 ~~~~ \forall  \psi,\phi\in H_{0}^{1}(\Omega), \label{beta}\\
    a(\phi,\phi)& \geq \delta |\phi|_1^2\;\;\;\;\;\;\;\;\;\;\;\; \forall \phi\in H_{0}^{1}(\Omega),  \label{alpha}
\end{align}
with $\beta,\delta \in \mathbb{R}^+$. Note that under the assumed condition on the coefficient matrix  $\{a_{ij}(x)\}_{1\leq i,j\leq d},$ the bilinear form corresponding to $A$ satisfies the above properties. 

 We additionally assume the following elliptic $H^2$-regularity property: for every
$g\in L^2(\Omega)$, the weak solution $z\in H_0^1(\Omega)$ of
\[
 a(\phi,z)=\langle g,\phi\rangle
 \qquad\text{for all }\phi\in H_0^1(\Omega)
\]
satisfies
\begin{equation}\label{eq:elliptic-H2}
 |z|_{H^2(\Omega)}\le C_{\mathrm{ell}}\lVert g\rVert_{L^2(\Omega)}.
\end{equation}
Here $C_{\mathrm{ell}}$ depends on the domain, but not on the meshes, time steps or the memory kernel.

\section{Semidiscrete scheme}\label{sec:3}
In this section, we derive \emph{a posteriori} error estimates for the spatially semidiscrete Galerkin approximations to the problem \eqref{ee}.
 Let  $ \mathcal{T}_h$ be  a conforming, shape-regular partition of $\Omega$ with the mesh-size function  $h$  defined as 
\begin{align*}
     h(x)=\text{diam}(K), ~~\text{where}~ K\in\mathcal{T}_h~~ \text{and}~ x\in K.
\end{align*} 
 Let $\mathcal{E}_h=\{E\}$ be the set of all internal sides of $\mathcal{T}_h$. These are edges in $d=2$ or faces in $d=3$ that are contained in the interior of $\Omega$. The union of all internal sides $\cup_{E\in\mathcal{E}_h}E$ is denoted by $\Sigma_h$.
 
Now, consider the continuous piecewise linear finite element space corresponding to $\mathcal{T}_h$ as 
 \begin{equation*}
     \mathbb{V}_h=\{\chi\in H_{0}^{1}(\Omega):\chi_{|K}\in\mathbb{P}_1(K)~\forall K\in \mathcal{T}_h\},
 \end{equation*}
where $\mathbb{P}_1(K)$ is the space of linear polynomials over $K$ in $d$ dimensions. However, the analysis extends to piecewise polynomials of any degree. The semidiscrete finite element approximation $u_h:[0,T]\rightarrow \mathbb{V}_h$ of $u$ is defined by 
 \begin{align}\label{SFEA}
      \langle u_{h,t},\phi_h \rangle+\int_{0}^{t} \mathcal{K}(t-s) a(u_h(s),\phi_h) ds &=\langle f,\phi_h\rangle~~~~~\forall \phi_h\in \mathbb{V}_h,\\
      u_h(\cdot,0) &=\mathcal{P}_{h}u_0,\nonumber
 \end{align}
   where $\mathcal{P}_{h}:L^2(\Omega)\rightarrow \mathbb{V}_h$ is the $L^2$-projection operator which is defined as
    \begin{equation}\label{3.4}
      \langle \mathcal{P}_{h}v,\phi_h\rangle=\langle v,\phi_h\rangle ~~~~\forall\phi_h\in \mathbb{V}_h.
  \end{equation}
Now, define the discrete elliptic operator $A_h:\mathbb{V}_h\rightarrow \mathbb{V}_h$ by
  \begin{equation} \label{DEO}
      \langle A_hv,\phi_h\rangle=a(v,\phi_h)~~~~\forall\phi_h\in \mathbb{V}_h.
  \end{equation}
Using the definitions above, the semidiscrete operator form of \eqref{SFEA} can be written as
  \begin{equation}\label{DF_S}
     u_{h,t}+\int_{0}^{t} \mathcal{K}(t-s) A_hu_h(s) ds=f_h,
 \end{equation}
  where $f_h=\mathcal{P}_{h}f$.
  
To construct the residuals, following \cite{LM06}, for a function $v\in \mathbb{V}_h$, the bilinear form can be represented as 
  \begin{equation}\label{3.2}
      a(v,\phi)= \langle A_{el}v,\phi\rangle + \langle J[v],\phi\rangle_{\Sigma_h}
      ~~~~\forall \phi\in H_{0}^{1}(\Omega),
  \end{equation}
  where $A_{el}v$ is the regular part of the distribution $-\nabla\cdot(\mathbf{A}\nabla v)$, which is defined as a piecewise continuous function (see, \cite[(25)]{LM06}) such that 
   \begin{equation}\label{Ael}
      \langle A_{el}v,\phi\rangle=\sum_{K\in \mathcal{T}_h}\langle-\nabla\cdot(\mathbf{A}\nabla v),\phi\rangle_K
      ~~~~\forall \phi\in H_{0}^{1}(\Omega),
  \end{equation}
  and $J[v]$ is the spatial jump of the field $\mathbf{A} \nabla v$ across an element side $E\in \mathcal{E}_h$ defined as 
  \begin{equation}\label{J_v}
      J[v]|_{E}(x)=[\mathbf{A}\nabla v]_E(x):=  \lim_{\kappa \to 0}(\mathbf{A}\nabla v(x-\kappa \nu_E )- \mathbf{A}\nabla v(x+\kappa  \nu_E ))\cdot \nu_E,
  \end{equation}
 where $\nu_E$ is a fixed choice of one of the two possible unit normal vectors to $E$.

Inspired by \cite{MN03,LM06}, we now introduce the elliptic reconstruction operator.
  \begin{definition}
  For a given $v\in \mathbb{V}_h$, define the elliptic reconstruction operator $ \mathcal{R}_h:\mathbb{V}_h\rightarrow H_{0}^{1}(\Omega)$ associated with the bilinear form $a(\cdot,\cdot)$ as
  \begin{equation}\label{ER}
      a(\mathcal{R}_hv,\phi)=\langle A_hv,\phi\rangle~~~~\forall\phi\in H_{0}^{1}(\Omega).
  \end{equation}
  \end{definition}
  \noindent
Note that the function $\mathcal{R}_hv$ is called the elliptic reconstruction of $v$. In view of the definition \eqref{ER}, the Galerkin orthogonality follows, i.e.,
  \begin{equation}\label{GO}
      a(\mathcal{R}_hv-v,\phi_h)=0~~~~\forall\phi_h\in \mathbb{V}_h.
  \end{equation}

In order to derive \emph{a posteriori} error bounds, we decompose the main error $e:=u_h-u$ as follows:
  \begin{equation}\label{splite}
      e=\rho-\epsilon, ~~~ where~~ \rho:=\mathcal{R}_hu_h-u,~~\epsilon:=\mathcal{R}_hu_h-u_h, 
  \end{equation}
where $\epsilon$ is referred to as the reconstruction error, while $\rho$ represents the temporal error. 

\noindent
Recall the interpolation error estimates.  
  \begin{lemma}[Interpolation inequalities \cite{S90}]\label{l4.1}
  Let  $\Pi_h:H_{0}^{1}(\Omega)\rightarrow \mathbb{V}_h$ be the Scott-Zhang interpolation operator. Then, for sufficiently smooth $\psi$ and for finite element space of polynomial degree $l$, the following interpolation inequalities hold for $j\leq l+1$  
  \begin{align}
      \|h^{-j}(\psi-\Pi_h\psi)\| &\leq C_{1,j}|\psi|_j, \\
      \|h^{1/2-j}(\psi-\Pi_h\psi)\|_{\Sigma_h} &\leq C_{2,j}|\psi|_j,
\end{align}
  where the constants $C_{1,j}$ and $C_{2,j}$ depend only on the shape-regularity of the family of triangulations.
\end{lemma}
The following lemma provides bounds on the elliptic reconstruction error $\epsilon$ and its time derivative. 
\begin{lemma}[Elliptic reconstruction error estimate]\label{l4.2}
For the time-dependent approximation $u_h (t) \in \mathbb{V}_h$, the following error estimates hold:
\begin{align}
     \|\epsilon(t)\| \leq \Theta(u_h(t))\;\;\;\;\;\\
     \mbox{and}\;\;\;\; \|\epsilon_t(t)\|\leq \Theta (u_{h,t}(t)),\label{T-RE}
\end{align}    
    where 
    \begin{align*}
     \Theta(u_h(t)) &:= \widetilde C_1\norm{h^2(A_{\mathrm{el}}u_h-A_hu_h)}+\widetilde C_2\norm{h^{3/2}J[u_h]}_{\Sigma_h},\\
     \Theta(u_{h,t}(t))&:= \widetilde C_1\norm{h^2(A_{\mathrm{el}}u_h-A_hu_h)_t}+\widetilde C_2\norm{h^{3/2}J[u_{h,t}]}_{\Sigma_h},
 \end{align*}
with the constants $\widetilde C_1$ and $\widetilde C_2$ depending on the elliptic regularity and interpolation constants.

\begin{proof}
The argument follows the structure of \cite[Lemma~2.3]{LM06}. For brevity, we omit the proof.
\end{proof}
\end{lemma}
For the temporal error $\rho$, we now derive an \emph{a posteriori} estimate in the following lemma.
 \begin{lemma}[Temporal error estimate] \label{l4.4}
 The following estimate holds for the temporal error $\rho=\mathcal{R}_hu_h-u:$
\begin{equation*}
  \norm{\rho(t)}\leq \norm{\rho(0)}+2\int_{0}^{t}\Big[\Theta(u_{h,t}(s))+\norm{(f_h-f)(s)}\Big]ds.    
\end{equation*}
 \end{lemma}
\begin{proof} 
  In view of \eqref{WF}, for all $\phi \in H_{0}^{1}(\Omega)$, we obtain
 \begin{align*}
       \langle\rho_t,\phi \rangle+\int_{0}^{t} \mathcal{K}(t-s) a(\rho(s),\phi) ds 
      =& \langle\mathcal{R}_hu_{h,t},\phi\rangle+\int_{0}^{t} \mathcal{K}(t-s) a(\mathcal{R}_hu_h(s),\phi) ds-\langle f,\phi\rangle.
 \end{align*}
Using \eqref{DF_S} and \eqref{ER}, we arrive at 
\begin{align}\label{ereq}
  \langle\rho_t,\phi \rangle+\int_{0}^{t} \mathcal{K}(t-s) a(\rho(s),\phi) ds 
  &=\langle\epsilon_t,\phi \rangle+\langle f_h-f,\phi \rangle.
\end{align}
Set $\phi=\rho(t)$ in \eqref{ereq} and use the Cauchy--Schwarz inequality to obtain
\begin{align} \label{4.14}
   \frac{1}{2}\frac{d}{dt}\norm{\rho(t)}^2 + \int_{0}^{t} \mathcal{K}(t-s)a(\rho(s),\rho(t))ds 
   &\leq (\norm{\epsilon_t}+\norm{f_h-f})\norm{\rho(t)}.
\end{align}
Now, integrate \eqref{4.14} from $0$ to $t$ and use \eqref{PK} to find that
\begin{equation*}
    \norm{\rho(t)}^2 \leq \norm{\rho(0)}^2 + 2\int_{0}^{t}\Big(\norm{\epsilon_t(s)}+\norm{(f_h-f)(s)}\Big)\norm{\rho(s)}ds.
\end{equation*}
Letting $t_0$ be such that $\norm{\rho(t_0)}=\sup_{s\in [0,t]}\norm{\rho(s)}$, we conclude 
\begin{equation*}
    \norm{\rho(t)}^2 \leq \norm{\rho(t_0)}^2 \leq  \bigg\{\norm{\rho(0)} + 2\int_{0}^{t}\bigg(\norm{\epsilon_t(s)}+\norm{(f_h-f)(s)}\bigg)ds\bigg\}\norm{\rho(t_0)},
\end{equation*}
and hence,
\begin{equation*}
    \norm{\rho(t)} \leq \norm{\rho(0)} + 2\int_{0}^{t}\bigg(\norm{\epsilon_t(s)}+\norm{(f_h-f)(s)}\bigg)ds.
\end{equation*}
Applying the estimate for \(\norm{\epsilon_t(s)}\) in Lemma \ref{l4.2}, we obtain the desired result. This completes the proof.
\end{proof}

Below, we state the main theorem of this section.
\begin{thm}[Semidiscrete $L^{\infty}(L^2)$-norm \emph{a posteriori} error estimate]\label{Thm:3.1}
 Let $u$ and $u_h$ satisfy \eqref{ee} and \eqref{SFEA}, respectively. Then, the following a posteriori error bound holds for $0\leq t\leq T:$
 \begin{equation}
      \norm{(u-u_h)(t)}\leq \norm{u(0)-u_h(0)} + \Theta(u_h(0)) + \Theta(u_h(t)) + 2 \int_{0}^{t}\Big[\Theta(u_{h,t}(s))+ \norm{(f_h-f)(s)}\Big]ds.
 \end{equation}
 \end{thm}
\begin{proof}  Decompose the error $e$ as in \eqref{splite} and use Lemmas \ref{l4.2} and \ref{l4.4} to conclude the rest of the proof.
\end{proof}
\noindent
Note that the final estimate in Theorem~\ref{Thm:3.1} is derived without invoking Gronwall's lemma, and the constants appearing in the estimate are independent of the final time $T$, the mesh parameters, and the kernel.
 \section{Analysis for the fully discrete scheme}\label{sec:4}
This section focuses on the fully discrete scheme based on the BE method.

For the temporal discretization, divide the interval $[0, T]$ into $N$ smaller  subintervals $I_n = (t_{n-1}, t_n]$, $n = 1, 2, \ldots, N$, in such a way that $t_0 = 0$ and $t_N = T$. The corresponding time step sizes are denoted by $\tau_n = t_n - t_{n-1}$. 
For convenience, we set $f^n(x) = f(x, t_n)$ to represent the function $f$ evaluated at time $t_n$. 

At each time level $t_n$, we associate a conforming, shape-regular triangulation $\mathcal{T}_n$ of $\Omega$, and define its elementwise mesh-size function \(h_n\) by
$$
h_n(x)=\operatorname{diam}(K),\;\;x\in K,\;\; K\in\mathcal{T}_n.$$
We now define the corresponding finite element space
\[
    \mathbb{V}^n = \left\{ \phi \in H_0^{1}(\Omega) : \phi|_K \in \mathbb{P}_1(K) \; \;\forall K \in \mathcal{T}_n \right\}.
\]
\noindent
Let \(\mathcal E_n\) represent the collection of interior sides of \(\mathcal T_n\), i.e., the edges in two dimensions or faces in three dimensions that lie in the interior of \(\Omega\). The union of all interior sides of \(\mathcal T_n\) is denoted by \(\Sigma_n\), i.e., $\Sigma_n:=\bigcup_{E\in\mathcal{E}_n}E$.
We assume that consecutive triangulations $\mathcal{T}_{n-1}$ and
$\mathcal{T}_n$ are compatible refinements of a conforming, uniformly
shape-regular common coarsening $\widehat{\mathcal{T}}_n$ such that
\[
 \mathbb{V}(\widehat{\mathcal{T}}_n)
 =\mathbb{V}^{n-1}\cap\mathbb{V}^{n},
\]
where $\mathbb{V}(\mathcal{T})$ denotes the continuous $\mathbb{P}_1$ finite element space on a triangulation $\mathcal{T}$.
For each $\widehat K\in\widehat{\mathcal{T}}_n$, put
$H_{\widehat K}:=\operatorname{diam}(\widehat K)$ and define
\begin{equation}\label{eq:common-mesh-size}
 \widehat h_n|_{\widehat K}:=H_{\widehat K}.
\end{equation}
 Let $\widehat\Sigma_n$ be the interior skeleton of $\widehat{\mathcal{T}}_n$, and let
\[
 \breve\Sigma_n:=\Sigma_{n-1}\cup\Sigma_n,
\]
where the union is understood as a geometric set. Then $\widehat\Sigma_n\subset\Sigma_{n-1}\cap\Sigma_n$, and the remaining sides
$\breve\Sigma_n\setminus\widehat\Sigma_n$ lie in the interiors of
common-coarsening elements.

On sides we set $\widehat h_n|_E:=\max\{H_{\widehat K}:E\subset\partial\widehat K\}$ for $E\subset\widehat\Sigma_n$, and $\widehat h_n|_E:=H_{\widehat K}$ for $E\subset\breve\Sigma_n\setminus\widehat\Sigma_n$, where $\widehat K\in\widehat{\mathcal{T}}_n$ is the parent containing $E$.
We further assume that there exists $L\ge1$, fixed independently of $n$
and of the individual mesh sizes, for which
\begin{equation}\label{eq:mesh-ratio}
 \frac{H_{\widehat K}}{h_K}\le L
 \quad\text{for every }\widehat K\in\widehat{\mathcal{T}}_n
 \text{ and every }K\subseteq\widehat K
 \text{ in }\mathcal{T}_{n-1}\cup\mathcal{T}_n,
\end{equation}
where $h_K=\operatorname{diam}(K)$.

In practice, if consecutive meshes are compatible and their relative
refinement depth is uniformly bounded, condition~\eqref{eq:mesh-ratio}
holds with a fixed $L$ for standard shape-regular refinement rules.
It follows that the additional-side trace estimate used in
Lemma~\ref{spatial} has a uniform constant $C_3$; this constant may
depend on $L$, but not on individual mesh sizes, time steps or the
memory kernel.

Let $q_n$ denote the positive type quadrature rule \cite{MT93} used to discretize the integral term, i.e.,
\begin{equation} 
    q_n(y)=\sum_{j=1}^{n} W_{nj}y^j \approx \int_{0}^{t_n}\mathcal{K}(t_n-s)y(s) ds, \label{quad1}
\end{equation}
where
\begin{equation}\label{Wnj}
W_{nj} = 
\begin{cases}
\tau_j \, \mathcal{K}(t_n - t_j) & \quad \text{for smooth kernels,} \\
\displaystyle \int_{t_{j-1}}^{t_j} \mathcal{K}(t_n - s)\, ds & \quad \text{for weakly singular kernel.} 
\end{cases}
\end{equation}
Note that for the weakly singular case, we use the piecewise-constant interpolant taking the value $y^j=y(t_j)$ in $(t_{j-1},t_j]$. 

For a given initial value $U^0 = \mathcal{P}_0^0 u(0)$, the fully discrete BE approximation is to 
find $U^n \in \mathbb{V}^n$, $n = 1, 2, \ldots, N$, such that  
\begin{equation} \label{FDS}
    \tau_n^{-1} \langle U^n - U^{n-1}, \phi_n \rangle 
    + q_n\!\big(a(U, \phi_n)\big) 
    = \langle f^n, \phi_n \rangle 
    \qquad \forall \phi_n \in \mathbb{V}^n,
\end{equation}
where $\mathcal{P}_0^n$ denotes the $L^2$-projection operator onto $\mathbb{V}^n$.

Furthermore, for $v \in \mathbb{V}^n$, the bilinear form $a(\cdot, \cdot)$ retains the same structure as in \eqref{3.2},  
\begin{equation}
    a(v, \phi) = \langle A_{el} v, \phi \rangle 
    + \langle J[v], \phi \rangle_{\Sigma_n} 
    \qquad \forall \phi \in H_0^{1}(\Omega),
\end{equation}
where $A_{el}$ and $J[v]$ are defined in \eqref{Ael} and \eqref{J_v}, respectively.

For the fully discrete error analysis, the elliptic reconstruction operator from \eqref{ER} is reformulated as  
\begin{equation} \label{ER(T)}
    a(\mathcal{R}^n v, \phi) = \langle A^n v, \phi \rangle 
    \qquad \forall \phi \in H_0^{1}(\Omega),
\end{equation}
for each $v \in \mathbb{V}^n$, where the discrete operator 
$A^n : \mathbb{V}^n \to \mathbb{V}^n$ is defined through  
\begin{equation} \label{DEOn}
    \langle A^n v, \phi_n \rangle = a(v, \phi_n) 
    \qquad \forall \phi_n \in \mathbb{V}^n.
\end{equation}
 We further introduce the elliptic projection operator $\mathcal{P}_1^n:H_0^1(\Omega) \rightarrow \mathbb{V}^n$ by 
\begin{equation}\label{EllProj}
    a(\mathcal{P}_1^n v, \phi_n) = a(v, \phi_n)
    \qquad \forall \phi_n \in \mathbb{V}^n.
\end{equation}

For each time interval $I_n$, we define a continuous-in-time approximation $U(t)$ 
that linearly interpolates the BE approximations $U^{n-1}$ and $U^n$:
\begin{equation} \label{U}
    U(t) := l_{n-1}(t) U^{n-1} + l_n(t) U^{n}, 
    \qquad \text{for } n = 1, \ldots, N,
\end{equation}
where, on $I_n$, the coefficients $l_{n-1}(t)$ and $l_n(t)$ are defined by
\begin{equation} \label{hat}
    l_{n-1}(t) = \frac{t_n - t}{\tau_n}, 
    \qquad 
    l_n(t) = \frac{t - t_{n-1}}{\tau_n}.
\end{equation}
Using this interpolation, we construct a time-dependent discrete elliptic reconstruction 
$\omega(t)$, given by
\begin{equation}\label{omega}
    \omega(t) := l_{n-1}(t)\mathcal{R}^{n-1}U^{n-1} 
    + l_n(t)\mathcal{R}^nU^n, 
    \qquad \text{for } \;t\in I_n \;\;\;\mbox{and}\;\;n = 1, \ldots, N,
\end{equation}
where $\mathcal{R}^n$ denotes the elliptic reconstruction operator defined in \eqref{ER(T)}. 

To facilitate the analysis that follows, we define the discrete temporal derivatives and their projections. 
For $n = 1, \ldots, N$, set
\begin{equation*}
    \partial U^n := \frac{U^n - U^{n-1}}{\tau_n}, 
    \qquad 
    \bar{\partial} U^n := \mathcal{P}_0^n \partial U^n
    = \frac{U^n - \mathcal{P}_0^n U^{n-1}}{\tau_n}.
\end{equation*}
Furthermore, the projected forcing term is given by
\begin{equation*}
    \bar{f}^n := \mathcal{P}_0^n f^n, \qquad n = 0, 1, \ldots, N.
\end{equation*}

 For any $\phi \in H_0^1(\Omega) $, 
by employing the $L^2$-projection, elliptic projection and the discrete operator,  together with the fully discrete scheme \eqref{FDS}, we obtain
\begin{align*}
   \langle \bar{\partial} U^n + q_n(A^n\mathcal{P}_1^n U)-\bar{f}^n, \phi\rangle&=\langle \bar{\partial} U^n + q_n(A^n\mathcal{P}_1^n U)-\bar{f}^n, \mathcal{P}_0^n\phi\rangle \\
   &=\langle \bar{\partial} U^n, \mathcal{P}_0^n\phi\rangle+ \langle q_n(A^n\mathcal{P}_1^n U), \mathcal{P}_0^n\phi\rangle-\langle \bar{f}^n, \mathcal{P}_0^n\phi\rangle \\
   &=\tau_n^{-1}\langle U^n - \mathcal{P}_0^n U^{n-1},\mathcal{P}_0^n\phi\rangle+q_n(\langle A^n\mathcal{P}_1^n U,\mathcal{P}_0^n\phi\rangle)- \langle \bar{f}^n, \mathcal{P}_0^n\phi\rangle \\
   &=\tau_n^{-1}\langle U^n -U^{n-1},\mathcal{P}_0^n\phi\rangle+q_n(a(\mathcal{P}_1^nU,\mathcal{P}_0^n\phi))-\langle f^n, \mathcal{P}_0^n\phi\rangle \\
   &=\tau_n^{-1}\langle U^n -U^{n-1},\mathcal{P}_0^n\phi\rangle+q_n(a(U,\mathcal{P}_0^n\phi))-\langle f^n, \mathcal{P}_0^n\phi\rangle \\ 
   &=0
\end{align*}
Thus, the fully discrete scheme \eqref{FDS} can be reformulated in the distributional form
\begin{equation} \label{DF}
    \bar{\partial} U^n + q_n(A^n\mathcal{P}_1^n U(x)) = \bar{f}^n(x), 
    \qquad \forall x \in \Omega.
\end{equation}

\begin{definition}[Error estimators]\label{def:est}
For $n=1,2,\cdots,N$ (and, for $\eta_n^{\mathrm{rec}}$, also for $n=0$), we introduce the following error estimators.\\
\medskip
\noindent\textbf{Elliptic reconstruction error estimator.}
We define the elliptic reconstruction error estimator by
\begin{equation}\label{rec}
\eta_n^{\mathrm{rec}}
:=
C_0\left(
\sum_{K\in\mathcal{T}_n}
h_K^4 \|R^n\|_K^2
+
\sum_{E\in\mathcal{E}_n}
h_E^3 \|J^n\|_E^2
\right)^{1/2},
\end{equation}
where  $R^n=A_{el}U^n-A^nU^n$ and $J^n=J[U^n]$, the inner residual 
and jump residual, respectively, and the constant $C_0$ depends only on $\Omega$ and the shape-regularity of \(\mathcal {T}_n\).

 \medskip 
 \noindent\textbf{Spatial error estimator.}
The spatial  error estimator is defined by
\begin{equation}\label{eta_S}
    \eta_n^S:=C_1\norm{\widehat{h}_n^{\,2}\partial R^n} + C_2\norm{\widehat{h}_n^{\,3/2}\partial J^n}_{\widehat{\Sigma}_n}+C_3\left\|\widehat{h}_n^{\,3/2}\partial J^n\right\|_{\breve{\Sigma}_n\setminus\widehat{\Sigma}_n}. 
\end{equation}
Here $\partial R^n:=(R^n-R^{n-1})/\tau_n$ and $\partial J^n:=(J^n-J^{n-1})/\tau_n$. The constants $C_1,C_2,C_3$ in $\eta_n^S$ 
depend only on the elliptic regularity constant and the shape-regularity of the meshes.

\medskip
\noindent\textbf{Temporal error estimator.}
The time discretization contribution is quantified by  
 \begin{equation}\label{lambda}
   \eta_n^T :=\frac{1}{\tau_n}\int^{t_n}_{t_{n-1}}\Big(\norm{\bar{f}^n - f^n}+ \norm{f^n-f(t)}\Big)~dt.
\end{equation}
\medskip
\noindent\textbf{Quadrature error estimator.}
The error induced by numerical quadrature is bounded by
\begin{align}\label{Qn}
 \eta_n^{Q}&:=
 \begin{cases}
      C_{Q}\sum_{j=1}^n \tau_j^2 \bigg(\norm{A^n\mathcal{P}_1^nU^{j-1}}+\norm{A^n\mathcal{P}_1^nU^{j}}+\norm{A^n\mathcal{P}_1^n \partial U^j}\bigg)  &~\text{ for smooth kernel},\\
     \sum_{j=1}^{n}\tau_j\,W_{nj}\, \norm{A^n\mathcal{P}_1^n\partial U^j}&~\text{ for weakly singular kernel},
      \end{cases}
\end{align}
 where $C_Q:=\max\{\gamma,\gamma_1\}$ is a computable kernel quantity and $W_{nj}$ are the weights \eqref{Wnj}.\\
\noindent\textbf{Memory consistency error estimator.} The error associated with the memory-lag and history mesh-change approximations is quantified by
\begin{equation}\label{eta_Q,2}
\eta_n^{K}:=\eta_n^L+\eta_n^H,
\end{equation}
where $\eta_n^L$ is the \emph{memory-lag estimator} and $\eta_n^H$ is the \emph{history mesh-change estimator}, given by
\begin{align}
\eta_n^L&:=\tau_n\cK_1(t_n)\max_{1\leq j\leq n}\big\|A^n\cP_1^n\partial U^j\big\|+\big(\cK_1(t_n)-\cK_1(t_{n-1})\big)\max_{0\leq j\leq n}\big\|A^n\cP_1^nU^j\big\|,\label{eq:etaL}\\
\eta_n^H&:=\cK_1(t_n)\max_{0\leq j\leq n}\big\|A^jU^j-A^n\cP_1^nU^j\big\|. \label{eq:etaH}
\end{align}
\medskip
\noindent\textbf{Mesh change error estimator.}
The contribution arising from mesh changes is given by
\begin{equation}\label{mu}
   \eta_n^M:=\norm{(\mathcal{P}_0^n-I)\tau_n^{-1}U^{n-1}}.
\end{equation}
\end{definition} 

We now derive the \emph{a posteriori} error bounds for the fully discrete BE scheme by decomposing the total error \(e = U-u \) as
\begin{equation}
    e = \rho - \epsilon,
\end{equation}
where \(\rho = \omega-u \) denotes the temporal error, and \(\epsilon = \omega-U \) denotes the elliptic reconstruction error.

\begin{lemma}[Elliptic reconstruction error estimate]\label{lERec}
The following error bound holds for the reconstruction error:
  \begin{equation*}
      \norm{\mathcal{R}^nU^n-U^n} \leq \eta_n^{rec} ,
  \end{equation*}
 where $\eta_n^{rec}$ is given by \eqref{rec}.
\end{lemma}
\begin{proof} The result mirrors  \cite[Lemma 2.3]{LM06}, and the details of the proof are therefore omitted for conciseness. 
\end{proof} 

We now proceed to derive the error estimate for \(\rho\). The starting point of our analysis is the following error equation satisfied by \(\rho\).
  \begin{lemma}[Error equation for \(\rho\)]
For each $t \in I_n$ and $\phi \in H_0^1(\Omega)$, the error equation is given by
\begin{equation}\label{PEq}
  \begin{split}
      \langle \rho_t,\phi \rangle &+ \int_0^t \mathcal{K}(t-s)a(\rho(s),\phi)~ds = \langle\epsilon_t,\phi \rangle + \langle \bar{f}^n-f,\phi \rangle + \langle  \tau_n^{-1} (\mathcal{P}_0^nU^{n-1}-U^{n-1}),\phi \rangle  \\
     &\quad \qquad + 
     \bigg< \int_{0}^{t_n} \mathcal{K}(t_n-s)A^n\mathcal{P}_1^nU(s)~ds-q_n(A^n\mathcal{P}_1^nU),\phi \bigg>\\
     &\quad \qquad +\int_{0}^{t} \mathcal{K}(t-s)a(\omega(s),\phi)~ds-\int_{0}^{t_n} \mathcal{K}(t_n-s)\langle A^n\mathcal{P}_1^nU(s),\phi \rangle~ds.\\
\end{split}
  \end{equation}
\end{lemma}
\begin{proof} For $t\in I_n$ and $\phi \in H_0^1(\Omega)$, using \eqref{WF}  and \eqref{DF}, we now obtain
\begin{align*}
     \langle \rho_t,\phi \rangle + \int_0^t \mathcal{K}(t-s)a(\rho(s),\phi)~ds &= \langle \omega_t,\phi \rangle +  \int_0^t \mathcal{K}(t-s)a(\omega(s),\phi)~ds - \langle f,\phi \rangle \\
     &= \langle\omega_t,\phi \rangle + \int_0^t \mathcal{K}(t-s)a(\omega(s),\phi)~ds- \langle \partial U^n, \phi \rangle \\
     &\qquad + \langle \tau_n^{-1} (\mathcal{P}_0^nU^{n-1}-U^{n-1}),\phi \rangle + \langle \bar{f}^n-f,\phi \rangle - \langle q_n(A^n\mathcal{P}_1^nU),\phi \rangle \\
     &= \langle\epsilon_t,\phi \rangle +   \langle \bar{f}^n-f,\phi \rangle  +  \langle \tau_n^{-1}(\mathcal{P}_0^nU^{n-1}-U^{n-1}),\phi \rangle \\
     &\qquad + 
     \bigg< \int_{0}^{t_n} \mathcal{K}(t_n-s)A^n\mathcal{P}_1^nU(s)~ds-q_n(A^n\mathcal{P}_1^nU),\phi \bigg>\\
     &\qquad  +  \int_0^t \mathcal{K}(t-s)a(\omega(s),\phi)~ds-\int_{0}^{t_n} \mathcal{K}(t_n-s)\langle A^n\mathcal{P}_1^nU(s),\phi \rangle~ds.
\end{align*}
This concludes the proof.
\end{proof}

The next lemma yields a bound on the temporal error $\rho(t)$.
\begin{lemma}[Bound for \(\rho\)]\label{l(PE)}
The following bound holds for \(\rho\):
    \begin{equation}
       \norm{\rho_m^*}^2\leq \norm{\rho(0)}^2 + 2I_m,
    \end{equation}
where 
\begin{equation*}
    \norm{\rho(t_m^*)}=\norm{\rho_m^*}:= \underset{0 \leq t \leq t_m}{\max} \norm{\rho(t)},~~~~~ t_m^*\in~[0,t_m]
\end{equation*}
and
\begin{equation*}
    I_m:=\sum_{n=1}^m(I_1^n+I_2^n+I_3^n+I_4^n+I_5^n)
\end{equation*}
with $I_1^n$ denoting the spatial error which is given by
\begin{equation*}
    I_1^n:=\int_{t_{n-1}}^{t_n} |\langle\epsilon_t(t),\rho(t) \rangle|~dt,
\end{equation*}
$I_2^n$ denotes the time discretization error and is given by
 \begin{equation*}
     I_2^n:=\int_{t_{n-1}}^{t_n} |\langle \bar{f}^n-f(t), \rho(t)\rangle|~dt\, 
 \end{equation*}
 $I_3^n$ denotes the quadrature error and is given by
 \begin{align*}
     I_3^n&:=\int_{t_{n-1}}^{t_n} \bigg| \bigg< \int_{0}^{t_n} \mathcal{K}(t_n-s)A^n\mathcal{P}_1^nU(s)~ds-q_n(A^n\mathcal{P}_1^nU),\rho(t)\bigg>\bigg|~dt
 \end{align*}
 $I_4^n$ denotes the memory consistency error and is given by 
\begin{align*}
    I_4^n&:=\int_{t_{n-1}}^{t_n} \bigg| \int_{0}^{t} \mathcal{K}(t-s)a(\omega(s),\rho(t))~ds-\int_{0}^{t_n} \mathcal{K}(t_n-s)\langle A^n\mathcal{P}_1^nU(s),\rho(t) \rangle~ds  \bigg|~dt
\end{align*}
and $I_5^n$ denotes the mesh-change error and is given by
 \begin{equation*}
     I_5^n:= \int_{t_{n-1}}^{t_n} |\langle \tau_n^{-1}(\mathcal{P}_0^nU^{n-1}-U^{n-1}), \rho(t) \rangle |~dt.
 \end{equation*}
\end{lemma}
\begin{proof} Set $\phi=\rho(t)$ in \eqref{PEq} to obtain
\begin{align*}
\frac{1}{2}\frac{d}{dt}\norm{\rho(t)}^2
    &+ \int_0^t \mathcal{K}(t-s)a(\rho(s),\rho(t))~ds \\
    &=\langle\epsilon_t,\rho(t) \rangle + \langle \bar{f}^n-f ,\rho(t) \rangle + \langle \tau_n^{-1} (\mathcal{P}_0^nU^{n-1}-U^{n-1}),\rho(t) \rangle  \\
     &\qquad + 
     \bigg< \int_{0}^{t_n} \mathcal{K}(t_n-s)A^n\mathcal{P}_1^nU(s)~ds-q_n(A^n\mathcal{P}_1^nU),\rho(t) \bigg>\\
     &\qquad +\int_{0}^{t} \mathcal{K}(t-s)a(\omega(s),\rho(t))~ds-\int_{0}^{t_n} \mathcal{K}(t_n-s)\langle A^n\mathcal{P}_1^nU(s),\rho(t) \rangle~ds\\
  &\leq |\langle \epsilon_t,\rho(t) \rangle | + |\langle \bar{f}^n-f ,\rho(t) \rangle| +|\langle\tau_n^{-1}( \mathcal{P}_0^nU^{n-1}-U^{n-1}),\rho(t) \rangle|  \\
  &\qquad 
  +\bigg|\bigg< \int_{0}^{t_n} \mathcal{K}(t_n-s)A^n\mathcal{P}_1^nU(s)~ds-q_n(A^n\mathcal{P}_1^nU),\rho(t) \bigg>\bigg|\\  
   &\qquad + \bigg|\int_{0}^{t} \mathcal{K}(t-s)a(\omega(s),\rho(t))~ds-\int_{0}^{t_n} \mathcal{K}(t_n-s)\langle A^n\mathcal{P}_1^nU(s),\rho(t) \rangle~ds    \bigg|.
\end{align*}
Now integrate with respect to $t$ from $t_{n-1}$ to $t_n$, then sum over $n=1,2,...,m$, and use \eqref{PK} to obtain the desired result. This completes the proof.
\end{proof}

We now estimate the terms $I_i^n,~i=1,2,3,4, 5$ appearing in Lemma \ref{l(PE)}. First, we estimate the spatial error in the following lemma.
\begin{lemma}[Spatial error estimate]\label{spatial}
Under the mesh assumptions \eqref{eq:common-mesh-size}--\eqref{eq:mesh-ratio}, there holds for $1\le m\le N$
\begin{equation}\label{eq:spatial-lemma-bound}
 \sum_{n=1}^m I_1^n
 \le \lVert\rho_m^*\rVert\sum_{n=1}^m\tau_n\eta_n^S,
\end{equation}
where $\eta_n^S$ is defined by \eqref{eta_S} with $\widehat h_n$ given
by~\eqref{eq:common-mesh-size} and  constants  in $\eta_n^S$ are independent of the individual mesh sizes, time steps and memory kernel; $C_3$ may depend on the fixed mesh-ratio bound $L$.
\end{lemma}
\begin{proof} The proof closely follows the approach outlined in \cite[Section 3.6]{LM06} that relies on elliptic regularity result \eqref{eq:elliptic-H2}, the Aubin–Nitsche trick along with the orthogonality property of the elliptic reconstructions. Thus, the details of the proof are omitted for brevity.
 \end{proof}
The following lemma provides a  bound for the error that arises due to time discretization.
 \begin{lemma}[Time error estimate]\label{Time_est}
   The following error bound holds for the time discretization error:
    \begin{equation}
     \sum_{n=1}^m I_2^n \leq   \norm{\rho_m^*} \sum_{n=1}^m\tau_n\eta_n^T,
    \end{equation}
    where $\eta_n^T$ is given by  \eqref{lambda}.
\end{lemma}
\begin{proof}   Using the Cauchy--Schwarz inequality, together with the splitting 
\[
\bar{f}^n - f(t) = \bar{f}^n - f^n + f^n - f(t),
\]
and subsequently, summing over $n = 1,2,\ldots,m$, we arrive at the desired result.
\end{proof}  
\begin{lemma} [Quadrature error estimate]\label{Quad_est}
 The following error bound holds for quadrature error:
 \begin{equation*}
      \sum_{n=1}^m I_3^n \leq \norm{\rho_m^*} \sum_{n=1}^m\tau_n \eta_n^{Q} ,
  \end{equation*} 
  where  $ \eta_n^{Q}$ is given by \eqref{Qn}.
\end{lemma}
\noindent
\emph{Proof.}
For each  $n = 1,2,\cdots,N$, recall that
\begin{align}
I_{3}^n := \int_{t_{n-1}}^{t_n} \bigg| \bigg< \int_{0}^{t_n} \mathcal{K}(t_n-s)A^n\mathcal{P}_1^nU(s)~ds-q_n(A^n\mathcal{P}_1^nU),\rho(t)\bigg>\bigg|~dt
\end{align}
We now proceed to estimate $I_{3}^n$. The term $I_{3}^n$ is estimated separately for the smooth and weakly singular kernels. The proof is divided into two steps. In the first step, we establish the smooth kernel estimate, while the second step is devoted to the weakly singular kernel.\\

\noindent
\textbf{Step 1: Smooth kernels}
\begin{sublemma}[Bound on $I_{3}^n$ for smooth kernels]\label{Quad_est_S}
\begin{equation}
\sum_{n=1}^m I_{3}^n\leq \norm{\rho_m^*} \sum_{n=1}^m\tau_n\eta_n^{Q}, 
\end{equation} 
where $\eta_n^{Q}$ is given by \eqref{Qn}.
\end{sublemma}
\begin{proof}
Upon setting $\psi_{j}(s)=t_{j-1}-s$, we obtain the following elementary result:  \begin{equation} \label{fact}
    \int^{t_j}_{t_{j-1}}y(s)ds-\tau_j y(t_{j})=  \int^{t_j}_{t_{j-1}}\psi_{j}(s)\frac{dy}{ds}ds.
\end{equation}
Substituting $y(s) = \mathcal{K}(t_n - s) A^n \mathcal{P}_1^n U(s)$ for $s \in I_j$, with $U(s) \in H_0^1(\Omega)$ in \eqref{fact}, summing over $j = 1$ to $n$, and subsequently invoking \eqref{quad1}, we arrive at 
\begin{align*}
I_{3}^n
&=\int_{t_{n-1}}^{t_n}\bigg|\sum_{j=1}^{n}  \int^{t_j}_{t_{j-1}} (s-t_{j-1})\mathcal{K}^{\prime}(t_n-s)  \langle A^n\mathcal{P}_1^nU(s), \rho(t) \rangle ~ds \\
&\qquad+ \sum_{j=1}^{n}  \int^{t_j}_{t_{j-1}} (t_{j-1}-s)\mathcal{K}(t_n-s)\langle A^n\mathcal{P}_1^n \partial U^j,\rho(t) \rangle ds \bigg|~dt.
\end{align*}
Invoking \eqref{kernel}, together with the Cauchy-–Schwarz and triangle inequalities, and using the bounds $|l_{j-1}(s)|,~|l_{j}(s)| \leq 1~\forall s \in I_j$ and $|s-t_{j-1}|\leq\tau_j$, we obtain
\begin{align*}
I_{3}^n&\leq 
 \gamma_1  \sum_{j=1}^{n}\tau_j \bigg(\int^{t_j}_{t_{j-1}} \norm{A^n\mathcal{P}_1^nU(s)}~ds\bigg)\bigg( \int_{t_{n-1}}^{t_n}\norm{\rho(t)}~dt\bigg) \\
 &\qquad+ \gamma \sum_{j=1}^{n}\tau_j \norm{ A^n\mathcal{P}_1^n \partial U^j}\bigg(\int^{t_j}_{t_{j-1}} 1~ds\bigg) \bigg(\int_{t_{n-1}}^{t_n}\norm{\rho(t)}~dt\bigg).
\end{align*}
With $C_{Q}$ as in \eqref{Qn}, summing over $n = 1,\cdots,m$ and using \eqref{Qn}, we obtain the asserted result. This completes Step 1.
\end{proof} 
\noindent
\textbf{Step 2: Weakly singular kernel}
\begin{sublemma}[Bound on $I_{3}^n$ for the weakly singular kernel]\label{Quad_est_WS}
\begin{equation*}
\sum_{n=1}^m I_{3}^n\leq \norm{\rho_m^*} \sum_{n=1}^m\tau_n\eta_n^{Q}, 
\end{equation*}
where $\eta_n^{Q}$ is given by \eqref{Qn}.
\end{sublemma}
\begin{proof}
In view of \eqref{quad1} and the identity $l_{j-1}(t) + l_{j}(t) \equiv 1$, we deduce
\begin{align*}
    I_{3}^n
   &=\int_{t_{n-1}}^{t_n}\bigg|\bigg<\sum_{j=1}^{n} \int^{t_j}_{t_{j-1}}\mathcal{K}(t_n-s)(A^n\mathcal{P}_1^nU(s)-A^n\mathcal{P}_1^nU^j)ds,\rho(t) \bigg>\bigg|~dt \\
   &=\int_{t_{n-1}}^{t_n}\bigg| \sum_{j=1}^{n} \int^{t_j}_{t_{j-1}}\mathcal{K}(t_n-s)\langle -\tau_j l_{j-1}(s) A^n\mathcal{P}_1^n\partial U^j,\rho(t) \rangle ds\bigg|~dt. 
\end{align*}
Invoking the Cauchy–-Schwarz inequality and noting that $\forall s \in I_j,~ |l_{j-1}(s)|\leq 1$, we arrive at
\begin{align*}
I_{3}^n
    &\leq \norm{\rho_m^*}\sum_{j=1}^{n}\tau_j\norm{A^n\mathcal{P}_1^n\partial U^j}\bigg(\int_{t_{n-1}}^{t_n}\int_{t_{j-1}}^{t_j}\frac{(t_n-s)^{\alpha-1}}{\Gamma(\alpha)}~ds~dt\bigg).
\end{align*}
Since the integrand does not depend on $t$, the double integral above equals $\tau_nW_{nj}$ by \eqref{Wnj}, and hence $I_{3}^n\leq\norm{\rho_m^*}\,\tau_n\eta_n^{Q}$. The desired result follows immediately by summing over $n=1, 2, \ldots, m$.
Further, since $\sum_{j=1}^nW_{nj}=\mathcal{K}_1(t_n)$, we have 
$$\eta_n^Q\leq\mathcal{K}_1(t_n)\max_{1\leq j\leq n}\tau_j\norm{A^n\mathcal{P}_1^n\partial U^j}.$$
\end{proof} 

\begin{lemma}[Memory consistency error estimate]\label{Quad_est_2}
\begin{equation*}
      \sum_{n=1}^m I_{4}^n\leq \norm{\rho_m^*} \sum_{n=1}^m\tau_n\eta_n^{K}, 
\end{equation*}
where $\eta_n^{K}$ is given by \eqref{eta_Q,2}.
\end{lemma}
\noindent
\emph{Proof.}  Recall that
\begin{align}
I_{4}^n := \int_{t_{n-1}}^{t_n} \bigg|\int_{0}^{t} \mathcal{K}(t-s)a(\omega(s),\rho(t))~ds-\int_{0}^{t_n} \mathcal{K}(t_n-s)\langle A^n\mathcal{P}_1^nU(s),\rho(t) \rangle~ds  \bigg|~dt.
\end{align} 

For $n=1,2,\cdots,N$, we introduce the piecewise linear functions $g_n$, and $\widehat g$ on $[0,t_n]$ by
\begin{equation}\label{eq:gn}
\begin{aligned}
g_n(s)&:=A^n\cP_1^nU(s)=l_{j-1}(s)A^n\cP_1^nU^{j-1}+l_j(s)A^n\cP_1^nU^j,\\
\widehat g(s)&:=l_{j-1}(s)A^{j-1}U^{j-1}+l_j(s)A^jU^j,
\end{aligned}
\end{equation}
for $s\in I_j$, $j=1,\dots,n$. 

\noindent By \eqref{ER(T)} and \eqref{omega}, for all $\phi\in H^1_0(\Omega)$, we have 
$$a(\omega(s),\phi)=\langle\widehat g(s),\phi\rangle.$$ 

\noindent Hence, we can write
\begin{equation}\label{eq:I32-split}
I_{4}^n=\int_{t_{n-1}}^{t_n}\big|\langle D_L(t)+D_H(t),\rho(t)\rangle\big|\,dt,
\end{equation}
where
\begin{align*}
D_L(t)&:=\int_0^t\cK(t-s)g_n(s)\,ds-\int_0^{t_n}\cK(t_n-s)g_n(s)\,ds,\\
D_H(t)&:=\int_0^t\cK(t-s)\big(\widehat g(s)-g_n(s)\big)\,ds.
\end{align*}
The first term represents the memory lag between $t$ and $t_n$, whereas the second term represents the history mesh change. We shall now estimate both terms in the subsequent Lemmas; the proof is completed after Lemma~\ref{lem:4.6.4}.

\begin{sublemma}[Memory-lag estimate]\label{lem:4.6.3}
For each $t\in I_n$,
\[
\|D_L(t)\|\leq\eta_n^L,
\]
 where $\eta_n^L$ is given by \eqref{eq:etaL}.
\end{sublemma}
\begin{proof}
We set
\[
\begin{gathered}
M_n:=\max_{0\leq j\leq n}\big\|A^n\cP_1^nU^j\big\|,\qquad
M_n':=\max_{1\leq j\leq n}\big\|A^n\cP_1^n\partial U^j\big\|.
\end{gathered}
\]
Since $g_n$ is continuous and piecewise linear on $[0,t_n]$ with $g_n(t_j)=A^n\cP_1^nU^j$ and $g_n'|_{I_j}=A^n\cP_1^n\partial U^j$, we have
\begin{equation}\label{eq:gn-bounds}
\|g_n(s)\|\leq M_n,\qquad\|g_n(s_1)-g_n(s_2)\|\leq|s_1-s_2|\,M_n',\qquad s,s_1,s_2\in[0,t_n].
\end{equation}
Substituting $r=t-s$ in the first integral of $D_L(t)$ and $r=t_n-s$ in the second one, and splitting the second integral at $r=t$, we obtain
\[
D_L(t)=\int_0^t\cK(r)\big[g_n(t-r)-g_n(t_n-r)\big]\,dr-\int_t^{t_n}\cK(r)\,g_n(t_n-r)\,dr.
\]
The first term compares the discrete history at two arguments that differ by $t_n-t\leq\tau_n$; this retains the cancellation between the memory at $t$ and at $t_n$. The second term is the memory accumulated over the short interval $(t,t_n]$. For $r\in[0,t]$, both $t-r$ and $t_n-r$ lie in $[0,t_n]$, and for $r\in[t,t_n]$, $t_n-r\in[0,t_n]$. Hence, by \eqref{eq:gn-bounds},
\begin{align*}
   \|D_L(t)\|&\leq(t_n-t)M_n'\int_0^t |\cK(r)| \,dr+M_n\int_t^{t_n} |\cK(r)|\,dr\\
&\leq\tau_n\cK_1(t)\,M_n'+\big(\cK_1(t_n)-\cK_1(t)\big)M_n.
\end{align*}
Since $\cK_1$ is nondecreasing, $\cK_1(t)\leq\cK_1(t_n)$ and $\cK_1(t_n)-\cK_1(t)\leq\cK_1(t_n)-\cK_1(t_{n-1})$ for $t\in I_n$. Hence, by \eqref{eq:etaL},
\[
\|D_L(t)\|\leq\eta_n^L.
 \]
This completes the proof.
\end{proof}
\begin{sublemma}[History mesh-change estimate]\label{lem:4.6.4}
For each $t\in I_n$, we have
\[
\|D_H(t)\| \leq \eta_n^H,
\]
where $\eta_n^H$ is given by \eqref{eq:etaH}.
\end{sublemma}
\begin{proof}
For $s \in I_j$, $1 \leq j \leq n$, it follows from \eqref{eq:gn} that 
\[
\widehat g(s)-g_n(s)=l_{j-1}(s)\big(A^{j-1}U^{j-1}-A^n\cP_1^nU^{j-1}\big)+l_j(s)\big(A^jU^j-A^n\cP_1^nU^j\big).
\]
Both brackets are of the form $A^iU^i-A^n\cP_1^nU^i$ with $0\leq i\leq n$. Since $l_{j-1},l_j\geq0$ and $l_{j-1}+l_j=1$, we get $\|\widehat g(s)-g_n(s)\|\leq \max_{0\leq j\leq n}\big\|A^jU^j-A^n\cP_1^nU^j\big\|$ for all $s\in[0,t_n]$. Therefore, for $t\in I_n$,
\begin{align*}
  \|D_H(t)\|&\leq \max_{0\leq j\leq n}\big\|A^jU^j-A^n\cP_1^nU^j\big\|\int_0^t |\cK(t-s)| \,ds\\
  &=\cK_1(t)\,\max_{0\leq j\leq n}\big\|A^jU^j-A^n\cP_1^nU^j\big\|\\
&\leq\cK_1(t_n)\,\max_{0\leq j\leq n}\big\|A^jU^j-A^n\cP_1^nU^j\big\|\\
&\leq \eta_n^H.
\end{align*}
\end{proof}
\noindent 
\textbf{Completion of the proof of Lemma \ref{Quad_est_2}.} Applying the  Cauchy--Schwarz inequality and the triangle inequality  to \eqref{eq:I32-split} and using the estimates established in Lemmas ~\ref{lem:4.6.3} and ~\ref{lem:4.6.4}, we obtain
\begin{align*}
I_{4}^n &\leq \int_{t_{n-1}}^{t_n}\norm{D_L(t)+D_H(t)}\norm{\rho(t)}\,dt \\
   &\leq \norm{\rho_m^*}\int_{t_{n-1}}^{t_n}\big(\norm{D_L(t)}+\norm{D_H(t)}\big)\,dt\\
  &\leq \|\rho_m^*\|\,\tau_n\eta_n^{K}.  
\end{align*}
Summing over $n=1, 2, \ldots, m$ completes the proof of Lemma ~\ref{Quad_est_2}.     \hfill $\square$
\noindent
\begin{lemma}[Mesh change estimate]\label{data_approx}
The following error bound holds for the mesh change error:
\begin{equation}
\sum_{n=1}^m I_5^n \leq \norm{\rho_m^*}\sum_{n=1}^m\tau_n\eta_n^M,
\end{equation}
where $\eta_n^M$ is given by \eqref{mu}.
\end{lemma}
\begin{proof}   
Applying  the Cauchy--Schwarz inequality, we obtain 
the desired error estimate.
\end{proof}  

\begin{lemma}[Error estimate for \(\rho\)] \label{est}
    For each $m = 1,2,\cdots, N$, the following estimate holds:
    \begin{equation}
        \underset{t\in[0,t_m]}{\max}\norm{\rho(t)} \leq \sqrt{2}\norm{\rho(0)} + 2\sigma_{1,m},
    \end{equation}
where
\begin{equation*}
    \sigma_{1,m}:=\sum_{n=1}^m\tau_n (\eta_n^S + \eta_n^T+ \eta_n^Q+ \eta_n^{K}+\eta_n^M ).
\end{equation*}
\end{lemma}
\begin{proof}  
Applying Lemmas \ref{spatial}-\ref{data_approx} in Lemma \ref{l(PE)} yields
\begin{align*}
   \norm{\rho_m^*}^2 \leq \norm{\rho(0)}^2 + 2\norm{\rho_m^*}\sigma_{1,m}. 
\end{align*}
Now applying Young's inequality, we obtain
\begin{align*}
   \norm{\rho_m^*}^2 \leq 2\norm{\rho(0)}^2 + 4\sigma_{1,m}^2.
\end{align*}
The assertion now follows by using the fact $a^2+b^2\leq (a+b)^2,~~ a,b \geq 0$, and this completes the proof.
\end{proof}  
\begin{thm}[Space-time $L^{\infty}(L^2)$-norm \emph{a posteriori} error estimate]\label{Final_Th}
Let $u$ satisfy \eqref{ee} and let $U$ be given by \eqref{U}. Then, for each $m = 1,2,\cdots,N$, the following error bound holds:
    \begin{equation}
        \underset{t\in[0,t_m]}{\max}\norm{u(t)-U(t)}\leq \sqrt{2}\norm{\mathcal{R}^0U^0-u(0)}+ \underset{0 \leq j \leq m}{\max}{\eta_j^{rec}}+ 2\sigma_{1,m},
    \end{equation}
    where $\eta_n^{rec}$ and $\sigma_{1,m}$ are defined in \eqref{rec} and Lemma \ref{est}, respectively.
\end{thm}
\begin{proof} 
In view of the splitting $e=\rho-\epsilon$ with $\rho=\omega-u$ and $\epsilon=\omega-U$, we obtain
\begin{equation}\label{error}
     \norm{e(t)}=\norm{U(t)-u(t)}\leq \norm{\epsilon(t)}+\norm{\rho(t)}.
 \end{equation}
 Now, for $t\in I_n$, we use Lemma \ref{lERec} to bound the first term
\begin{align}
    \norm{\ep(t)}&=\norm{l_{n-1}(t)\ep^{n-1}+l_n(t)\ep^n}
     \leq \underset{0 \leq j \leq m}{\max}\norm{\ep^j}\nonumber\\
    &\leq \underset{0 \leq j \leq m}{\max}{\eta_j^{rec}} \label{ep}.
 \end{align}
Combining \eqref{ep}, \eqref{error}, and Lemma \ref{est} completes the proof.
\end{proof}  
\begin{myrem}[$\alpha$-robustness of the upper bound]\label{rem:alpha}
For the weakly singular kernel case, the upper bound of Theorem~\ref{Final_Th} is robust with respect to $\alpha$ in the following sense.

\smallskip
\noindent
(i) The reliability constants are independent of $\alpha$.
The proof of Lemma~\ref{l(PE)} uses \eqref{PK},
so no $\alpha$-dependent coercivity constant enters. 
The constants $C_0,\ldots,C_3$ may depend on the elliptic regularity
constant, mesh shape-regularity, and the fixed relative-refinement
bound $L$, but remain independent of the individual mesh sizes,
time steps and fractional parameter $\alpha$.

\smallskip
\noindent
(ii) The kernel-dependent quantities admit uniform bounds.
In the weakly singular case, $\alpha$ enters the estimators \eqref{Qn}, \eqref{eq:etaL}
and \eqref{eq:etaH} only through the weights $W_{nj}$
and $\mathcal{K}_1(t)=t^{\alpha}/\Gamma(\alpha+1)$. Set
\[
 c_*:=\max_{1\le s\le2}\frac{1}{\Gamma(s)}<1.13,
 \qquad T_*:=\max\{1,T\}.
\]
Since for $0<\alpha<1$, $\Gamma(\alpha+1)\ge c_*^{-1}$ and $t^{\alpha}\le T_*$ for
$t\in[0,T]$, we have
\[
 \mathcal{K}_1(t)\le c_*T_*,
 \qquad
 \sum_{j=1}^{n}W_{nj}=\mathcal{K}_1(t_n)\le c_*T_*,
 \qquad
 \sum_{n=1}^m\big(\mathcal{K}_1(t_n)-\mathcal{K}_1(t_{n-1})\big)
 =\mathcal{K}_1(t_m)\le c_*T_*.
\]
Thus the $\alpha$-dependent factors in $\eta_n^{Q}$, $\eta_n^L$ and
$\eta_n^H$ are bounded uniformly for $\alpha\in(0,1)$.

For a fixed time grid, these factors and the quadrature weights \eqref{Wnj}
also have finite limits at both ends of $(0,1)$:
\[
 \lim_{\alpha\to0^+}W_{nj}=\delta_{nj},\quad
 \lim_{\alpha\to0^+}\mathcal{K}_1(t_n)=1\ (n\ge1),\quad
 \lim_{\alpha\to0^+}\big(\mathcal{K}_1(t_n)-\mathcal{K}_1(t_{n-1})\big)=\delta_{n1},\quad
\]
\[
 \lim_{\alpha\to1^-}W_{nj}=\tau_j,\quad
 \lim_{\alpha\to1^-}\mathcal{K}_1(t_n)=t_n,\quad
 \lim_{\alpha\to1^-}\big(\mathcal{K}_1(t_n)-\mathcal{K}_1(t_{n-1})\big)=\tau_n,\quad
\]
where $\delta_{nj}$ is the Kronecker delta. In particular, $\eta_n^Q\to\tau_n\norm{A^n\mathcal{P}_1^n\partial U^n}$ as $\alpha\to0^+$ and $\eta_n^Q\to\sum_{j=1}^n\tau_j^2\norm{A^n\mathcal{P}_1^n\partial U^j}$ as $\alpha\to1^-$.
\end{myrem}

\begin{myrem}[Nonconvex domains]\label{rem:nonconvex}
In Theorem~\ref{Final_Th}, the convexity of $\Omega$ enters only through
\eqref{eq:elliptic-H2}. On a nonconvex polygon with largest interior angle
$\omega$, \eqref{eq:elliptic-H2} holds with $|z|_{H^{1+s}(\Omega)}$,
$0<s<\pi/\omega$, in place of $|z|_{H^2(\Omega)}$ \cite{G85}. Using
Lemma~\ref{l4.1} with $j=1+s$, the proof of Theorem~\ref{Final_Th} carries
over with the weights $h^2$, $h^{3/2}$ in \eqref{rec}--\eqref{eta_S} replaced
by $h^{1+s}$, $h^{1/2+s}$. The L-shaped experiment in
Section~\ref{sec:num-adapt} uses the original weights.
\end{myrem}

\section{Numerical experiments}\label{sec:5}
We examine the convergence of the estimators for smooth and weakly singular kernels, their use in space--time adaptivity, and their dependence on the fractional parameter $\alpha$. 
We use FreeFEM and MATLAB for all the computations. Throughout, we use conforming $\mathbb{P}_1$ finite elements and BE time stepping. The estimators are computed up to the unevaluated constants in the reconstruction and interpolation estimates, which are set to one in the implementation. The reported values therefore assess convergence and practical performance, rather than provide certified upper bounds.

\subsection{Convergence on uniform meshes}\label{sec:num-uniform}
The first two experiments are posed on $\Omega=(0,1)^2$ with homogeneous Dirichlet boundary conditions, $A=-\Delta$ and final time $T=1$. We use uniform spatial meshes and uniform time steps $\tau_n=\tau$, with $\tau\sim h^2$ to balance the expected spatial and temporal errors. These coupled refinements assess the combined convergence behaviour; separate spatial and temporal studies would be needed to isolate the two orders.

 \subsubsection{Smooth kernel}
For the smooth kernel case, we adopt the benchmark presented in \cite{MT93}, where the problem was originally considered in one spatial dimension. We extend this to a two-dimensional setting, where the kernel is chosen as \(\mathcal{K}(t-s) = e^{-2(t-s)}\). The exact solution to the problem \eqref{ee} in this setting is given by
\begin{equation*}
  \begin{split}
    u(x,y,t) &= e^{-t}
    \left[\cos(\mu_{1}t) + \mu_{1}^{-1} \sin(\mu_{1}t)\right] \sin(\pi x)\sin(\pi y) \\
    &\quad + \frac{1}{4\pi^2} \left\{ 2 - e^{-t} \left[ 2\cos(\mu_{2}t) - (\mu_{2} - \mu_{2}^{-1})\sin(\mu_{2}t) \right] \right\} \sin(2\pi x)\sin(2\pi y),
  \end{split}
\end{equation*}
where \( \mu_1 = \sqrt{\pi^2 - 1} \) and \( \mu_2 = \sqrt{4\pi^2 - 1} \).

The corresponding source term \( f(x,y,t) \) is computed by substituting the exact solution into equation \eqref{ee}, yielding
\begin{equation*}
  \begin{split}
    f(x,y,t) &= 2 \sin(2\pi x)\sin(2\pi y) 
    + \pi^2 \mu_1^{-1} e^{-t} \sin(\mu_1 t) \sin(\pi x)\sin(\pi y) \\
    &\quad - e^{-t} \left[ \cos(\mu_2 t) + \mu_2^{-1} \sin(\mu_2 t) \right] \sin(2\pi x)\sin(2\pi y).
  \end{split}
\end{equation*}

\subsubsection{Weakly singular kernel}
Following \cite{KPF15}, we consider the weakly singular kernel defined by \(\mathcal{K}(t-s) = \frac{(t-s)^{\alpha-1}}{\Gamma(\alpha)} \), with \( \alpha = 1/2 \). The source term is chosen so that the exact solution of \eqref{ee} is
\begin{equation*}
    u(x,y,t) = (t^{3/2} + 1) \sin(\pi x)\sin(\pi y).
\end{equation*}
For $A=-\Delta$ and $\alpha=1/2$, substitution gives
\[
f(x,y,t)=\left[\frac32 t^{1/2}+2\pi^2\left(\frac{\Gamma(5/2)}{\Gamma(3)}t^2+\frac{t^{1/2}}{\Gamma(3/2)}\right)\right]\sin(\pi x)\sin(\pi y).
\]
In both square-domain examples, $u_0$ is obtained by setting $t=0$ in the exact solution.

\subsubsection{Error and estimator convergence}

At the final time ($m=N$), we compare the error in $L^\infty(0,T;L^2(\Omega))$ with the reconstruction, spatial, temporal, quadrature and memory-lag contributions:
\[
\max_{0\leq n\leq N}\eta_n^{\mathrm{rec}},\qquad
\sum_{n=1}^{N}\tau_n\eta_n^S,\qquad
\sum_{n=1}^{N}\tau_n\eta_n^T,\qquad
\sum_{n=1}^{N}\tau_n\eta_n^Q,\qquad
\sum_{n=1}^{N}\tau_n\eta_n^L.
\]
\begin{figure}[t]
\centering
\begin{subfigure}[t]{0.49\textwidth}
\centering
\includegraphics[width=\linewidth]{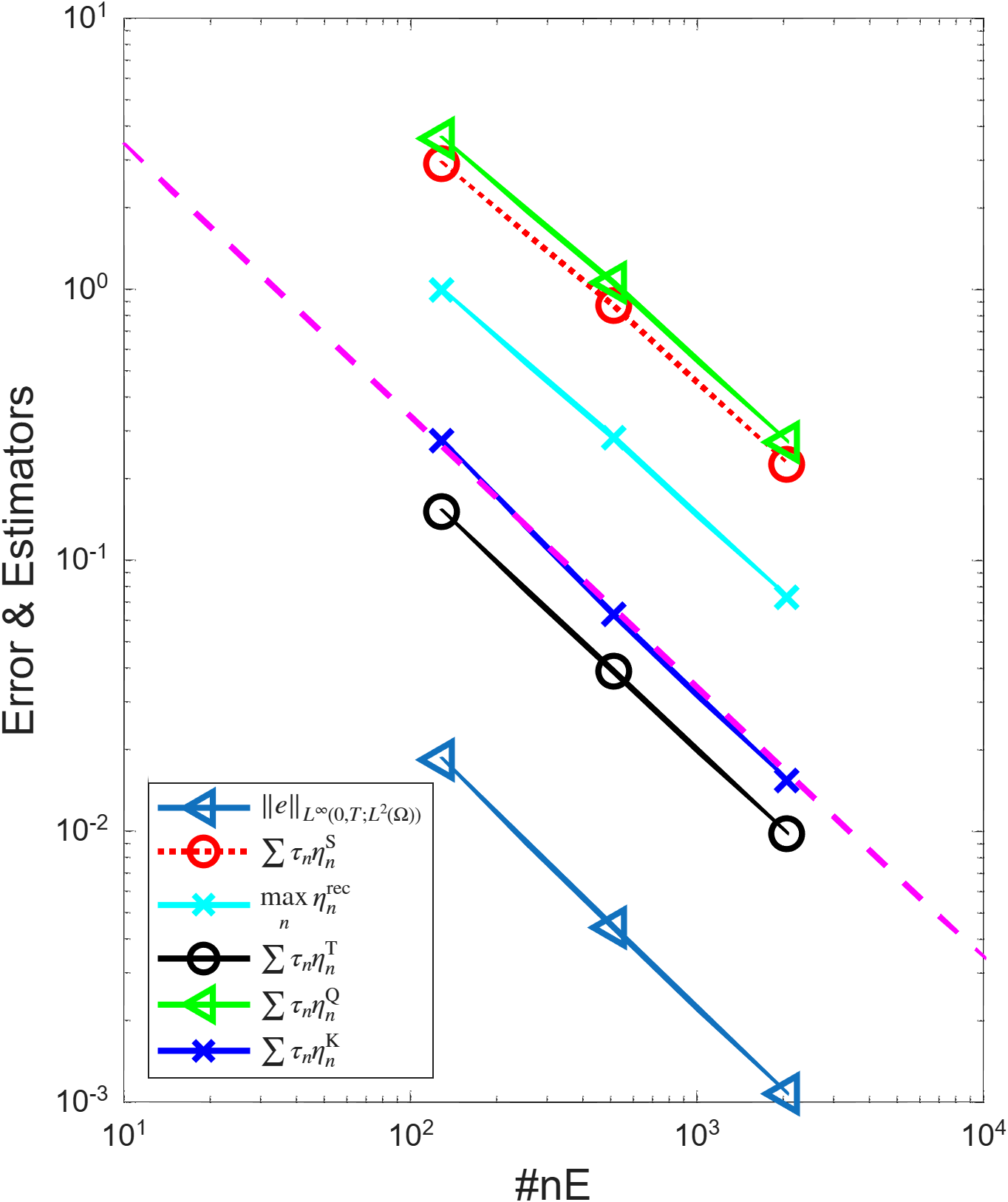}
\caption{Errors and estimators for smooth kernel}
\label{fig:BE_smooth_kernel_Square}
\end{subfigure}
\hfill
\begin{subfigure}[t]{0.49\textwidth}
\centering
\includegraphics[width=\linewidth]{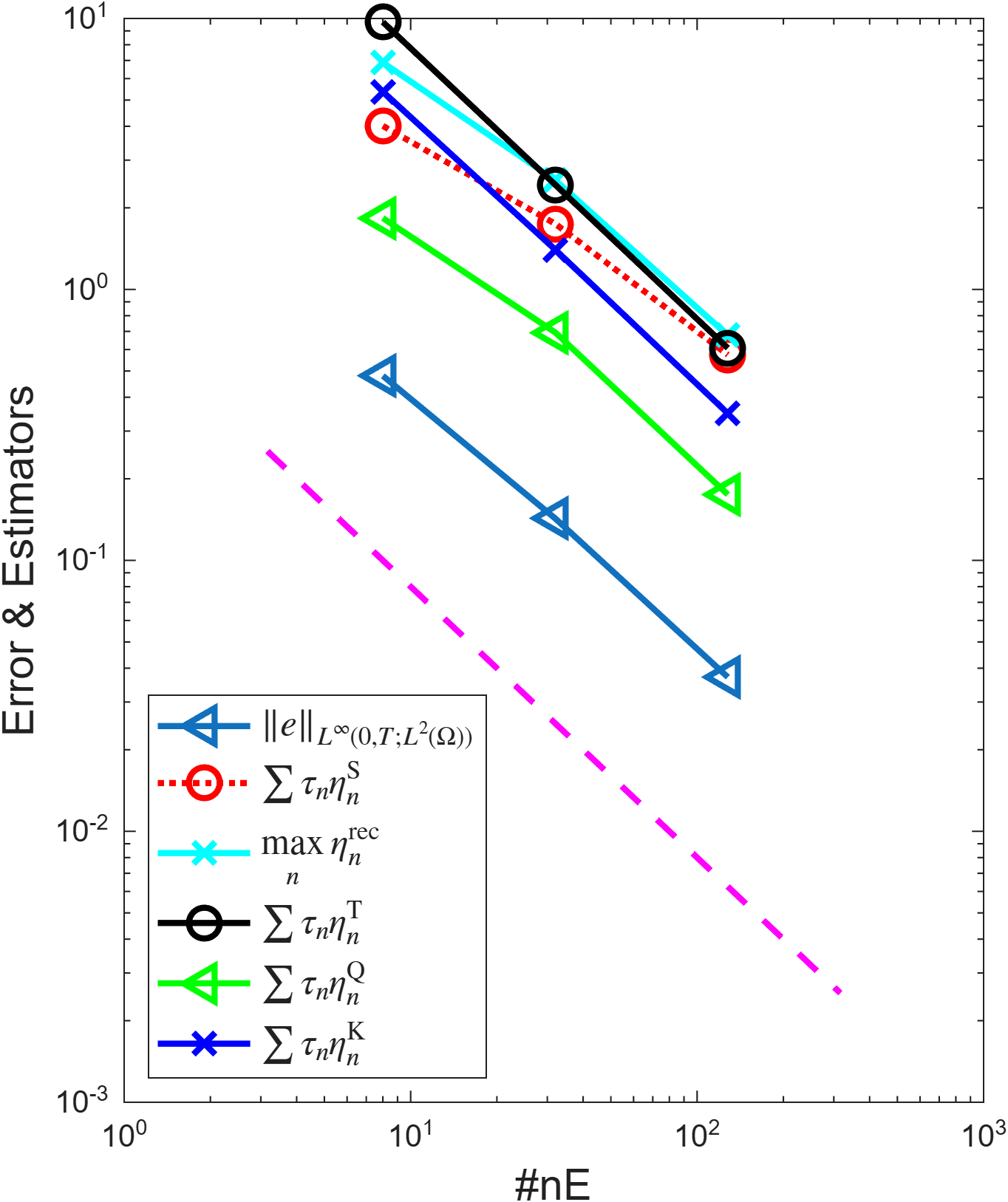}
\caption{Errors and estimators for weakly singular kernel}
\label{fig:BE_WS_kernel_Square}
\end{subfigure}
\caption{Error and estimator contributions on uniform meshes for (a) the smooth kernel and (b) the weakly singular kernel. The horizontal axis is the number of elements $\#nE$. The magenta dashed line has slope $-1$, corresponding to $O(h^2)$ under the coupling $\tau\sim h^2$. On each fixed spatial mesh, $\eta_n^H=\eta_n^M=0$ and $\eta_n^K=\eta_n^L$.}
\label{fig:convergence}
\end{figure}

Figure~\ref{fig:convergence} shows an approximately first-order decay with respect to the number of elements for the error and the estimator contributions in both examples. Since $\#nE\sim h^{-2}$, this behaviour is consistent with the expected $O(h^2+\tau)$ convergence under $\tau\sim h^2$. The different error magnitudes cannot be attributed solely to the kernel, since the two experiments use different exact solutions and source terms.

\subsection{Space--time adaptivity on an L-shaped domain}\label{sec:num-adapt}
\subsubsection{Benchmark problem}

The uniform-mesh tests do not assess adaptive refinement. We therefore consider a manufactured solution with a corner singularity and two temporal layers on an L-shaped domain. The L-shaped domain is chosen to introduce a spatial corner singularity and, together with the two temporal layers, to assess the practical performance of the space–time adaptive algorithm under reduced solution regularity. Since this domain is not convex, the regularity assumption \eqref{eq:elliptic-H2} does not hold, and Theorem~\ref{Final_Th} does not apply directly. This experiment complements the convergence tests on the convex domain in Section~\ref{sec:num-uniform} by examining the behaviour of the proposed indicators beyond the elliptic-regularity setting of the analysis; it is therefore presented as a numerical investigation of adaptivity rather than a verification of the reliability bound on nonconvex domains.

We prescribe the exact solution on the L-shaped domain $\Omega=(0,1)^2\setminus[0.5,1]^2$:
\begin{equation*}
u(x,t)=g(t)\,\Phi(x),\qquad \Phi(x)=r^{2/3}\sin\Big(\frac{2\phi}{3}\Big)\chi(r),\qquad g(t)=t^{3/4}+(t-1/2)_+^{3/4},
\end{equation*}
where $r=\sqrt{(x-0.5)^2+(y-0.5)^2}$ is the distance from the re-entrant corner $c=(0.5,0.5)$, $\phi\in[0,3\pi/2]$ is the angle measured from the upper re-entrant edge, and $\chi(r)$ is a $C^4$ radial cutoff that equals $1$ for $r\le 0.05$ and $0$ for $r\ge 0.45$. The kernel is $\mathcal{K}(t)=t^{-1/2}/\Gamma(1/2)$, i.e., $\alpha=1/2$, and $A=-\Delta$. The source term is
\begin{equation*}
f(x,t)=g'(t)\,\Phi(x)+(\mathcal{K}*g)(t)\,A\Phi(x),\qquad (\mathcal{K}*g)(t)=\frac{\Gamma(7/4)}{\Gamma(9/4)}\Big(t^{5/4}+(t-1/2)_+^{5/4}\Big),
\end{equation*}
with $A\Phi=-\sin(2\phi/3)\big(r^{2/3}\chi''+\tfrac{7}{3}r^{-1/3}\chi'\big)$; $A\Phi$ is supported in the cutoff annulus and vanishes for $r\le0.05$ and $r\ge0.45$, so $f$ vanishes on $\partial\Omega$.
For the evaluation of $f$, the contribution $\frac34(t-\frac12)^{-1/4}$ to $g'$ is included only when $t>1/2$ and is set to zero otherwise. Thus, $f$ is not evaluated at $t = 1/2$, even when this is a time node; the BE scheme evaluates $f$ only at positive time levels.

\noindent\textbf{Spatial singularity.} The domain has a re-entrant corner with an interior angle of $3\pi/2$. Near this corner, the solution behaves like $\Phi\sim r^{2/3}$ and its gradient behaves like $\nabla\Phi\sim r^{-1/3}$, which blows up at the corner. This is the classical corner singularity of non-convex domains; the resulting loss of regularity can reduce the convergence rate of finite element approximations on uniform meshes. The spatial adaptivity is designed to resolve it by refining the mesh near the corner.

\noindent\textbf{Temporal singularity.} The time-dependent factor $g(t)=t^{3/4}+(t-1/2)_+^{3/4}$ has a singular derivative at $t=0$ (where $g'(t)\sim t^{-1/4}$) and at $t=1/2$ (where $g'$ becomes unbounded as $t \rightarrow 1/2$). These are two temporal layers, which the time adaptivity must resolve by reducing the time step near these points. Together, the spatial and temporal singularities make this a challenging test for the 
space--time adaptive algorithm.

\subsubsection{Adaptive indicators and algorithm}
We use local indicators derived from the estimator components to select the triangulations $\mathcal{T}_n$ and the time steps $\tau_n$. The spatial tests control reconstruction and data oscillation, while a practical temporal indicator controls step selection.

\smallskip
\noindent\emph{Space indicators.} For $K\in\mathcal{T}_n$, we use the elementwise contributions of the elliptic reconstruction estimator \eqref{rec},
\begin{equation*}
\eta_{n,K}^2:=h_K^4\|R^n\|_K^2+\frac12\sum_{E\in\mathcal{E}_n,\,E\subset\partial K}h_E^3\|J^n\|_E^2,
\quad\text{and}\quad \big(\eta_n^{\mathrm{rec}}\big)^2=C_0^2\sum_{K\in\mathcal{T}_n}\eta_{n,K}^2,
\end{equation*}
and the elementwise data oscillation $\mathrm{osc}_{n,K}:=\|\bar f^n-f^n\|_K$, which is the first part of the time estimator \eqref{lambda}. Both quantities are reduced by refining the elements on which these estimators are large.

\smallskip
\noindent\emph{Time indicator.} The temporal indicator combines data variation, the current discrete time derivative weighted by the remaining memory, and the increment of the kernel primitive:

\begin{equation*}
\begin{split}
E_n:={}&\frac{1}{\tau_n}\int_{t_{n-1}}^{t_n}\|f^n-f(t)\|\,dt\\
&{}+\tau_n\,\mathcal{K}_1(T-t_{n-1})\,\|A^n\mathcal{P}_1^n\partial U^n\|+\big(\mathcal{K}_1(t_n)-\mathcal{K}_1(t_{n-1})\big)\|A^n\mathcal{P}_1^nU^n\|,
\end{split}
\end{equation*}
where $\mathcal{K}_1(T-t_{n-1})$ is the remaining integrated kernel weight. The nominal acceptance test is $TE_n\leq\mathrm{TOL}_t$. If every step satisfies this test, then $\sum_n\tau_nE_n\leq\mathrm{TOL}_t$, where $\mathrm{TOL}_t$ is the tolerance related to temporal contributions.

\begin{algorithm}[t]

\caption{Space--time adaptive BE scheme}\label{alg:adapt}
\begin{algorithmic}[1]
\Require initial mesh $\mathcal{T}_0$, initial step $\tau_1$, tolerances $\mathrm{TOL}_{\mathrm{rec}},\mathrm{TOL}_{\mathrm{osc}},\mathrm{TOL}_t$, marking parameter $\theta\in(0,1]$, threshold $\vartheta>0$
\State $U^0:=\mathcal{P}_0^0u_0$, $t_0:=0$, $n:=1$, $\mathcal{T}_1:=\mathcal{T}_0$
\While{$t_{n-1}<T$}
 \State $\tau_n:=\min\{\tau_n,T-t_{n-1}\}$, $t_n:=t_{n-1}+\tau_n$, $E^{\mathrm{prev}}:=\infty$
 \State \textbf{solve:} compute $U^n\in\mathbb{V}^n$ from \eqref{FDS}\label{alg:solve}
 \State \textbf{space:} \textbf{if} $\eta_n^{\mathrm{rec}}>\mathrm{TOL}_{\mathrm{rec}}\|U^n\|$ or $\|\bar f^n-f^n\|>\mathrm{TOL}_{\mathrm{osc}}\|f^n\|$ \textbf{then}
 \Statex \hspace{\algorithmicindent}\hspace{\algorithmicindent} mark a minimal set $\mathcal{M}\subset\mathcal{T}_n$ with $\sum_{K\in\mathcal{M}}\eta_{n,K}^2\geq\theta\sum_{K\in\mathcal{T}_n}\eta_{n,K}^2$
 \Statex \hspace{\algorithmicindent}\hspace{\algorithmicindent} (and, likewise, with $\mathrm{osc}_{n,K}^2$ for the second criterion);
 \Statex \hspace{\algorithmicindent}\hspace{\algorithmicindent} refine the marked elements to obtain a new $\mathcal{T}_n$ and go to line~\ref{alg:solve}
 \State \textbf{time:} compute $E_n$; \textbf{if} $T E_n>\mathrm{TOL}_t$ and $\tau_nE_n>\vartheta\,\mathrm{TOL}_t$ and $\tau_nE_n<0.9\,E^{\mathrm{prev}}$ \textbf{then}
 \Statex \hspace{\algorithmicindent}\hspace{\algorithmicindent} $E^{\mathrm{prev}}:=\tau_nE_n$, $\tau_n:=\tau_n\max\{\frac14,\,0.9\,\mathrm{TOL}_t/(TE_n)\}$, $t_n:=t_{n-1}+\tau_n$, and go to line~\ref{alg:solve}
 \State \textbf{accept:} compute $\eta_n^{\mathrm{rec}},\eta_n^S,\eta_n^T,\eta_n^Q,\eta_n^K,\eta_n^M$ and store $U^n$
 \State $\tau_{n+1}:=\tau_n\min\{2,\max\{\frac12,\,0.9\,\mathrm{TOL}_t/(TE_n)\}\}$, $\mathcal{T}_{n+1}:=\mathcal{T}_n$, $n:=n+1$
\EndWhile
\State \Return $U$ and the estimator $\max_{n}\eta_n^{\mathrm{rec}}+2\sigma_{1,N}$ of Theorem~\ref{Final_Th}
\end{algorithmic}
\end{algorithm}

Each trial step is solved using the stored history. If either spatial criterion fails, the marked elements are refined and the step is recomputed. The meshes are nested and are not coarsened, so $\mathcal{P}_0^nU^{n-1}=U^{n-1}$ and $\eta_n^M=0$. 

The time step controller limits reductions to a factor of at least $1/4$ and increases to at most a factor of $2$. Near $t=0$, the data-variation term in $E_1$ scales like $\tau_1^{-1/4}$, whereas its integrated contribution scales like $\tau_1^{3/4}$. The algorithm therefore also accepts a step when $\tau_nE_n\leq\vartheta\,\mathrm{TOL}_t$, or stops reducing it when the integrated indicator fails to decrease by at least $10\%$. The latter safeguard detects stagnation.

\subsubsection{Implementation and reported quantities}

 Algorithm~\ref{alg:adapt} is implemented in FreeFEM with $\mathbb{P}_1$ elements. Marked triangles are subdivided by \texttt{splitmesh}, so that the meshes are nested. We take the initial mesh with $h=1/8$, $\tau_1=10^{-4}$, $\theta=1/2$, $\vartheta=0.01$, the unevaluated constants set to $1$, and the tolerances
\[
(\mathrm{TOL}_{\mathrm{rec}},\mathrm{TOL}_{\mathrm{osc}},\mathrm{TOL}_t)=2^{-k/2}\,(0.12,\,0.12,\,0.2),\qquad k=0,1,2,
\]
which we refer to as the runs A0, A1 and A2, respectively. For comparison, the problem is also solved on uniform meshes with $h=1/8, 1/16, 1/32$ and uniform time steps $\tau=8h^2$, which we refer to as the runs U8, U16, and U32. Thus, the number in the label is $k$ for the adaptive runs and $1/h$ for the uniform runs; see the second column of Table~\ref{tab:adapt}. The reported error $\|e\|$ is the maximum of $\|u(t)-U(t)\|$ sampled at the time levels and interval midpoints, using spatial quadrature of order $5$. It approximates, but need not equal, the continuous-in-time $L^\infty(0,T;L^2(\Omega))$ error, and the estimator is $\eta:=\max_n\eta_n^{\mathrm{rec}}+2\sigma_{1,N}$ of Theorem~\ref{Final_Th}, with $\eta_n^Q$, $\eta_n^L$ and $\eta_n^H$ as in \eqref{Qn}, \eqref{eq:etaL} and \eqref{eq:etaH}. In Table~\ref{tab:adapt}, $N$ is the number of time steps, $\#\mathrm{dof}$ the number of degrees of freedom of the final mesh, and $W=\sum_n\#\mathrm{dof}_n$ the cumulative number of unknowns over accepted time steps. 
The experimental orders of convergence with respect to $W$,
\[
\mathrm{EOC}_e:=\frac{\log(\|e\|_i/\|e\|_{i-1})}{\log(W_i/W_{i-1})},\qquad \mathrm{EOC}_\eta:=\frac{\log(\eta_i/\eta_{i-1})}{\log(W_i/W_{i-1})},\qquad W:=\sum_n\#\mathrm{dof}_n,
\]
are computed between consecutive runs $i-1$ and $i$ of each group.

\begin{table}[htbp]
\centering
\caption{Adaptive runs A0, A1, A2 (Algorithm~\ref{alg:adapt}, tolerances $2^{-k/2}(0.12,0.12,0.2)$, $k=0,1,2$) and uniform runs U8, U16, U32 ($h=1/8,1/16,1/32$, $\tau=8h^2$) for the L-shaped benchmark with $\alpha=1/2$: number of time steps $N$, rejected steps, final number of degrees of freedom, work $W=\sum_n\#\mathrm{dof}_n$, the sampled error $\|e\|$, estimator $\eta=\max_n\eta_n^{\mathrm{rec}}+2\sigma_{1,N}$, and their experimental orders of convergence with respect to $W$.}\label{tab:adapt}
\small
\begin{adjustbox}{max width=\linewidth}
\begin{tabular}{llrrrrrrrr}
\hline
run & parameters & $N$ & rej. & $\#\mathrm{dof}$ & $W$ & $\|e\|$ & $\mathrm{EOC}_e$ & $\eta$ & $\mathrm{EOC}_\eta$\\
\hline
A0  & adaptive, $k=0$   & 120 & 2  & 2160 & 2.59e{+}5 & 1.01e{-}3 & --      & 1.562 & --\\
A1  & adaptive, $k=1$   & 165 & 2  & 3005 & 4.96e{+}5 & 7.47e{-}4 & $-0.47$ & 1.101 & $-0.54$\\
A2  & adaptive, $k=2$   & 229 & 1  & 4355 & 9.97e{+}5 & 5.54e{-}4 & $-0.43$ & 0.801 & $-0.46$\\
U8  & uniform, $h=1/8$  & 8   & -- & 71   & 5.68e{+}2 & 2.80e{-}2 & --      & 13.33 & --\\
U16 & uniform, $h=1/16$ & 32  & -- & 253  & 8.10e{+}3 & 7.98e{-}3 & $-0.47$ & 3.387 & $-0.52$\\
U32 & uniform, $h=1/32$ & 128 & -- & 933  & 1.19e{+}5 & 2.46e{-}3 & $-0.44$ & 0.898 & $-0.49$\\
\hline
\end{tabular}
\end{adjustbox}
\end{table}

\begin{table}[htbp]
\centering
\caption{Contributions to the estimator $\eta=\max_n\eta_n^{\mathrm{rec}}+2\sigma_{1,N}$ for the runs of Table~\ref{tab:adapt} ($\eta_n^H=\eta_n^M=0$ in all runs), and the share of the memory-lag estimator in $\eta$.}\label{tab:contrib}
\small
\begin{adjustbox}{max width=\linewidth}
\begin{tabular}{lrrrrrr}
\hline
run & $\max_n\eta_n^{\mathrm{rec}}$ & $\sum_n\tau_n\eta_n^S$ & $\sum_n\tau_n\eta_n^T$ & $\sum_n\tau_n\eta_n^Q$ & $\sum_n\tau_n\eta_n^L$ & $2\sum_n\tau_n\eta_n^L/\eta$\\
\hline
A0  & 0.0089 & 0.0143 & 0.0929 & 0.0883 & 0.5810 & 74\%\\
A1  & 0.0066 & 0.0107 & 0.0707 & 0.0628 & 0.4031 & 73\%\\
A2  & 0.0045 & 0.0074 & 0.0513 & 0.0447 & 0.2947 & 74\%\\
U8  & 0.3468 & 0.5226 & 2.9952 & 1.0668 & 1.9052 & 29\%\\
U16 & 0.1080 & 0.1679 & 0.5142 & 0.3121 & 0.6454 & 38\%\\
U32 & 0.0285 & 0.0440 & 0.1136 & 0.0783 & 0.1989 & 44\%\\
\hline
\end{tabular}
\end{adjustbox}
\end{table}

\begin{figure}[htbp]
\centering
\includegraphics[width=\linewidth]{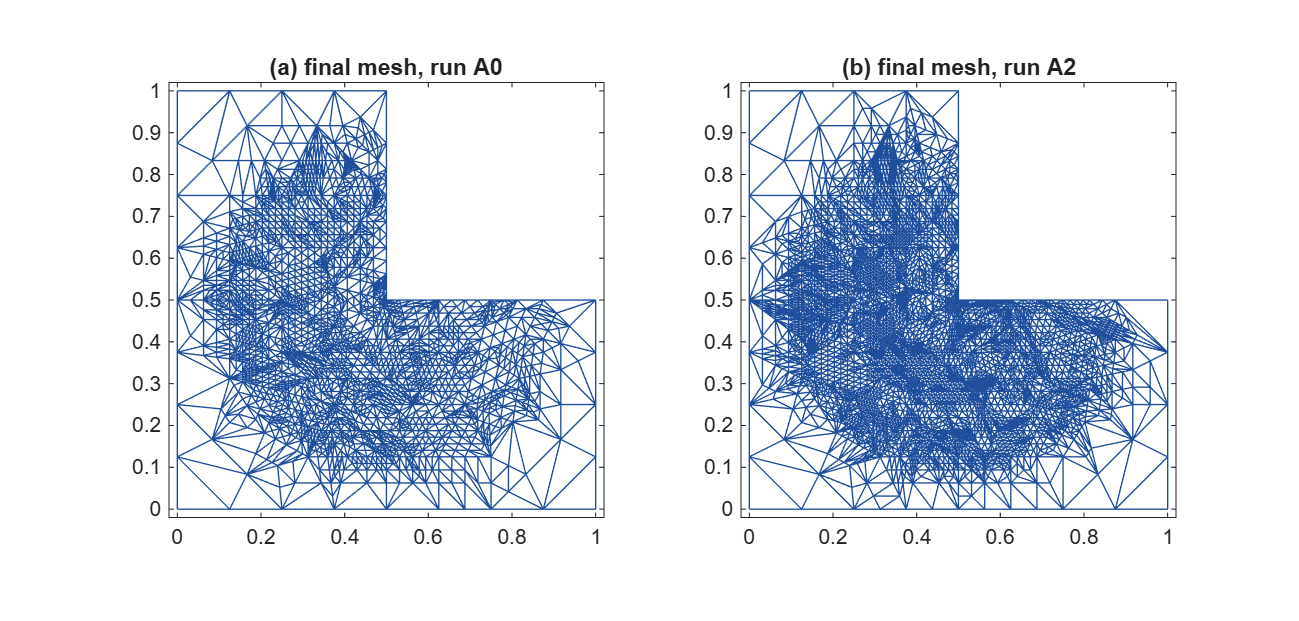}
\caption{Final adaptive meshes: (a) run A0 (4264 triangles), (b) run A2 (8644 triangles).}\label{fig:Lshape}
\end{figure}

\begin{figure}[htbp]
\centering
\includegraphics[width=0.5\linewidth]{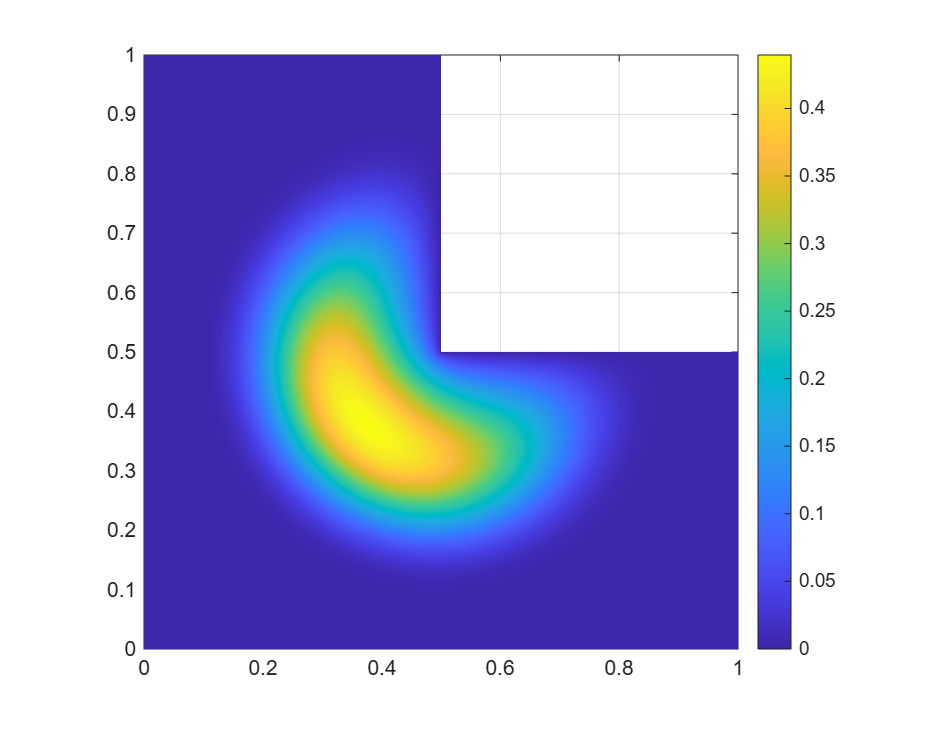}
\caption{Computed solution $U^N$ at $t=T=1$, run A2.}\label{fig:Lshape_sol}
\end{figure}

\begin{figure}[htbp]
\centering
\includegraphics[width=\linewidth]{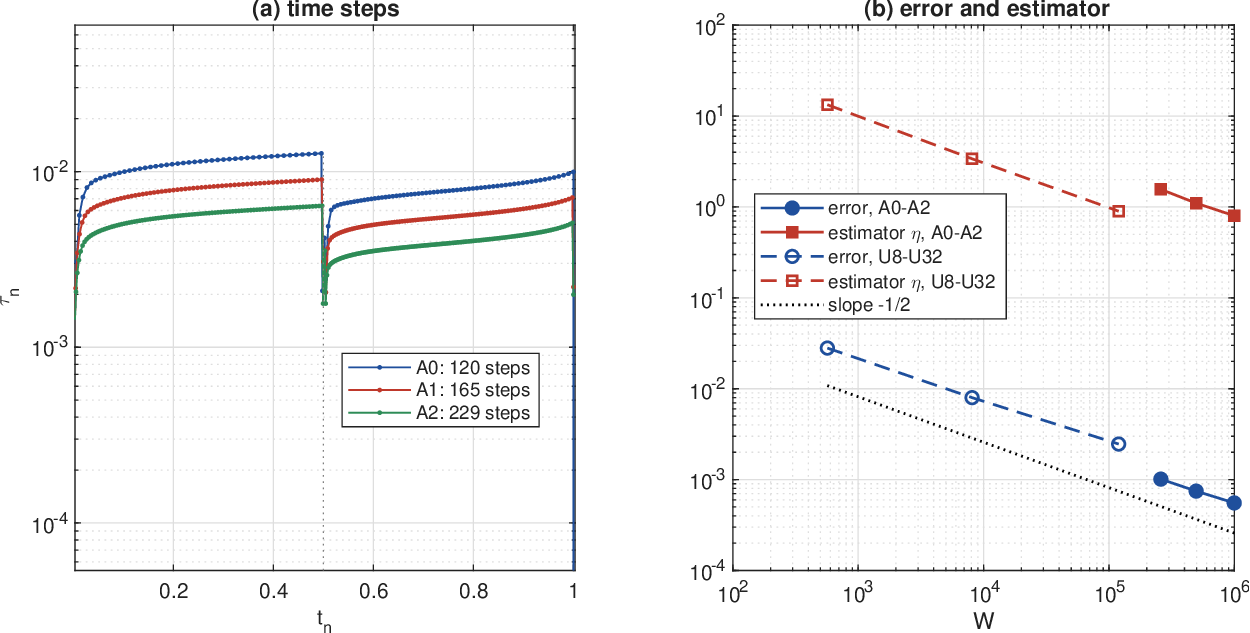}
\caption{L-shaped benchmark, $\alpha=1/2$. (a) Time step sizes $\tau_n$ of Algorithm~\ref{alg:adapt} for the runs A0, A1, A2; the temporal layer occurs at $t=1/2$. (b) Error and estimator $\eta$ versus the work $W=\sum_n\#\mathrm{dof}_n$ for the adaptive runs (filled markers) and the uniform runs (open markers); the dotted line has slope $-1/2$.}\label{fig:adapt}
\end{figure}

\subsubsection{Adaptive meshes, time steps and convergence} All mesh refinement takes place during the first time step; afterwards the meshes are not changed, so that $\#\mathrm{dof}$ in Table~\ref{tab:adapt} is also the number of unknowns in every time step and $W=N\cdot\#\mathrm{dof}$ for every run. Since the mesh is not changed after the first time step and $U^0=0$, the history mesh-change estimator $\eta_n^H$ vanishes, and $\eta_n^M=0$ by nestedness. The final meshes (Figure~\ref{fig:Lshape}) are refined most strongly at the re-entrant corner and in the ring $0.05<r<0.45$ where the cutoff derivatives contribute to $A\Phi$, whereas they remain coarse outside the support of $\Phi$. The grading at the corner is moderate, since the $r^{2/3}$ singularity affects the error in the $L^2(\Omega)$-norm only mildly; most of the refinement serves to resolve the data in the ring, triggered by the data oscillation $\mathrm{osc}_{n,K}$. The computed solution at $t=T$ is shown in Figure~\ref{fig:Lshape_sol}.

The time steps in Figure~\ref{fig:adapt}(a) are graded at both temporal layers. They start at $\tau_1=10^{-4}$ and grow by a factor of at most $2$ per step away from $t=0$. At $t=1/2$, where $g'$ becomes unbounded again, the algorithm detects the new layer only through $E_n$ and reduces the step sharply, after which the steps grow again. Also, each reduction of the tolerances by $1/\sqrt2$ increases the number of time steps by a factor of about $1.4$ ($N=120$, $165$, $229$; Table~\ref{tab:adapt}).

The adaptive runs in Table~\ref{tab:adapt} give successive error slopes $-0.47$ and $-0.43$, and estimator slopes $-0.54$ and $-0.46$, with respect to $W$. Across A0--A2, the slopes are approximately $-0.45$ and $-0.50$. The uniform runs give similar slopes, $\mathrm{EOC}_e=-0.47$ and $-0.44$ and $\mathrm{EOC}_\eta=-0.52$ and $-0.49$. Here, the slope $-1/2$ in Figure~\ref{fig:adapt}(b) is a reference scaling. For sufficiently regular solutions, the model $\|e\|\sim h^2+\tau$, together with $\tau\sim h^2$, gives $W\sim h^{-4}$ in two dimensions and hence $\|e\|\sim W^{-1/2}$. Thus, the adaptive sequence does not
improve the convergence rate with respect to $W$, but it improves the error level: A0 attains the error
$1.01\cdot10^{-3}$ with $W=2.59\cdot10^5$, whereas U32 attains $2.46\cdot10^{-3}$ with $W=1.19\cdot10^5$,
i.e., a $2.4$ times smaller error with about $2.2$ times the work.

We observe that the effectivity indices $\eta/\|e\|$ in Table~\ref{tab:adapt} are approximately 
$1.4$--$1.5\times10^3$ for the adaptive runs and $3.7$--$4.8\times10^2$ for the uniform runs. These values should be interpreted with care: the constants in the estimates are not evaluated but set to one, so the effectivity index is determined only up to these constants, and computable (sharper) values of them would change its magnitude.
In addition, Table~\ref{tab:contrib} shows that the memory-lag term contributes $73$--$74\%$ of the adaptive estimator; the estimator $\eta_n^L$ in \eqref{eq:etaL} retains large discrete derivatives from the initial layer and makes this contribution conservative. Nevertheless, the estimator decreases at nearly the same rate as the
error, which is consistent with the observed adaptive performance.

\subsection{Dependence on \texorpdfstring{$\alpha$}{alpha}}\label{sec:num-alpha}
To illustrate Remark~\ref{rem:alpha} numerically, we repeat the L-shaped experiment for the kernels $\mathcal{K}(t)=t^{\alpha-1}/\Gamma(\alpha)$ with $\alpha\in\{0.001,0.01,0.1,0.5,0.9,0.99,0.999\}$, i.e., values close to both limits $\alpha\to0^+$ and $\alpha\to1^-$ together with three intermediate values; $\alpha=0.5$ corresponds to the same nominal parameter choice as Table~\ref{tab:adapt}. The exact solution $u=g(t)\Phi(x)$ is the same as above; only the source term changes: $f=g'\Phi+(\mathcal{K}*g)A\Phi$ with $(\mathcal{K}*g)(t)=\frac{\Gamma(7/4)}{\Gamma(7/4+\alpha)}\big(t^{3/4+\alpha}+(t-\frac12)_+^{3/4+\alpha}\big)$. The tolerances and other algorithmic parameters are kept fixed, and for each $\alpha$ we compute the adaptive runs A0 and A2 (tolerances $(0.12,0.12,0.2)$ and $(0.06,0.06,0.1)$) and the uniform runs U16 and U32. The $\alpha=0.5$ runs in Table~\ref{tab:alpha} use different spatial meshes from those in Table~\ref{tab:adapt}, which accounts for the differences in the errors and estimators, although the time step counts coincide ($N=120$ for A0 and $N=229$ for A2).

\begin{table}[htbp]
\centering
\small
\caption{Dependence on $\alpha$: number of time steps $N$, error $\|e\|$ and estimator $\eta$ of the adaptive runs A0 and A2, and the reduction factors of the estimator and of the error from A0 to A2 (tolerances halved) and from U16 to U32 (mesh size halved).}\label{tab:alpha}
\begin{adjustbox}{max width=\linewidth}
\begin{tabular}{c|ccc|ccc|cc|cc}
\hline
 & \multicolumn{3}{c|}{A0} & \multicolumn{3}{c|}{A2} & \multicolumn{2}{c|}{A0$\to$A2} & \multicolumn{2}{c}{U16$\to$U32}\\
$\alpha$ & $N$ & $\|e\|$ & $\eta$ & $N$ & $\|e\|$ & $\eta$ & $\eta$ & $\|e\|$ & $\eta$ & $\|e\|$\\
\hline
0.001 & 124 & 7.61e{-}4 & 1.882 & 239 & 3.61e{-}4 & 0.959 & 0.51 & 0.47 & 0.26 & 0.32\\
0.01 & 125 & 7.76e{-}4 & 1.844 & 240 & 3.61e{-}4 & 0.950 & 0.52 & 0.46 & 0.26 & 0.32\\
0.1 & 126 & 9.75e{-}4 & 1.747 & 240 & 4.79e{-}4 & 0.904 & 0.52 & 0.49 & 0.26 & 0.32\\
0.5 & 120 & 1.08e{-}3 & 1.663 & 229 & 5.52e{-}4 & 0.803 & 0.48 & 0.51 & 0.27 & 0.31\\
0.9 & 103 & 1.13e{-}3 & 1.823 & 199 & 5.94e{-}4 & 0.895 & 0.49 & 0.53 & 0.27 & 0.31\\
0.99 & 100 & 1.36e{-}3 & 1.785 & 190 & 6.88e{-}4 & 0.917 & 0.51 & 0.51 & 0.27 & 0.32\\
0.999 & 99 & 1.09e{-}3 & 1.778 & 188 & 6.46e{-}4 & 0.924 & 0.52 & 0.59 & 0.27 & 0.32\\
\hline
\end{tabular}
\end{adjustbox}
\end{table}

\begin{figure}[htbp]
\centering
\includegraphics[width=0.85\linewidth]{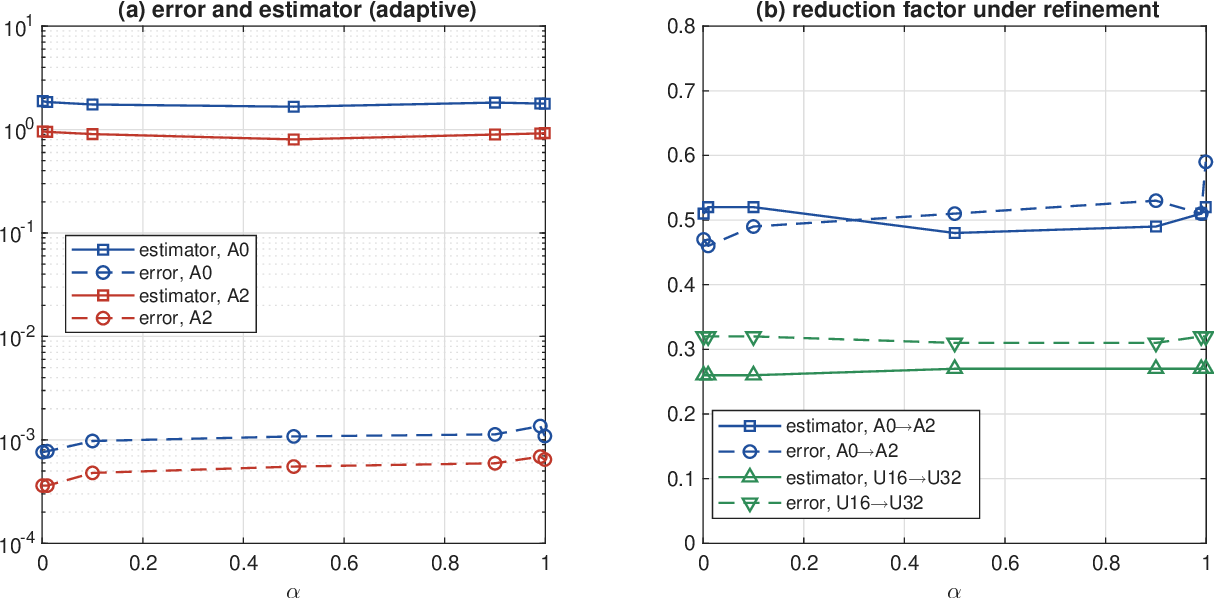}
\caption{Dependence on $\alpha$: (a) error and estimator of the adaptive runs A0, A2; (b) reduction factors of the estimator and of the error under refinement (A0$\to$A2 and U16$\to$U32), i.e., the last four columns of Table~\ref{tab:alpha}.}\label{fig:alpha}
\end{figure}

Table~\ref{tab:alpha} and Figure~\ref{fig:alpha} show the following:
\begin{enumerate}[label=(\alph*),leftmargin=*]
\item For the seven tested values of $\alpha$, the estimator remains within a narrow range: it lies between $1.66$ and $1.88$ for A0 and between $0.80$ and $0.96$ for A2, also for $\alpha=0.001$ and $\alpha=0.999$. Moreover, the estimators for $\alpha=0.001$ and $\alpha=0.01$ ($1.882$ and $1.844$), and those for $\alpha=0.99$ and $\alpha=0.999$ ($1.785$ and $1.778$), almost coincide, in accordance with the finite limits of the weights as $\alpha\to0^+$ and $\alpha\to1^-$ in Remark~\ref{rem:alpha}.
\item The estimator has similar reduction factors for all tested values of $\alpha$. Halving the tolerances (A0$\to$A2) reduces it by the factor $0.48$--$0.52$ and the error by $0.46$--$0.59$; halving the mesh size of the uniform runs (U16$\to$U32) reduces the estimator by $0.26$--$0.27$ and the error by $0.31$--$0.32$, for all tested values of $\alpha$ (last four columns of Table~\ref{tab:alpha}, and Figure~\ref{fig:alpha}(b)).
\item The error of A0 lies between $7.61\cdot10^{-4}$ and $1.36\cdot10^{-3}$, i.e., it changes by less than a factor of two over $\alpha\in[0.001,0.999]$, and the number of time steps of A0 lies between $99$ and $126$; it decreases from $126$ to $99$ as $\alpha$ increases from $0.1$ to $0.999$. The same tolerances give comparable errors and time step counts for the tested values. 
\end{enumerate}

\section{Conclusion}\label{sec:6}
In this article, we have established optimal-order \emph{a posteriori} error estimates for evolution equations with positive-type memory (both smooth and weakly singular kernels) in the $L^{\infty}(0,T;L^{2}(\Omega))$-norm, for both the spatially semidiscrete and fully discrete backward Euler schemes. 
The derived \emph{a posteriori} error estimates not only enhance the understanding of the accuracy of numerical solutions, but also contribute to the development of more reliable computational algorithms for problems exhibiting memory effects. Furthermore, the utilization of the elliptic reconstruction technique in the context of positive-type memory evolution equations represents a novel approach, extending the applicability of this powerful tool to a broader class of mathematical models. A key ingredient of the fully discrete analysis is the memory-lag estimate, which retains the cancellation between the memory terms at $t$ and at $t_n$ and yields an estimator of first order in time that depends on the kernel only through $\mathcal{K}_1$. The derived upper bounds are robust: the constants in the upper bounds are independent of the final time, the mesh size, the time steps, and the kernel; no exponential growth in time occurs; and, for the weakly singular kernel, the $\alpha$-dependent factors in the estimators remain uniformly bounded in $\alpha\in(0,1)$. The estimators also drive a space--time adaptive algorithm which, for a solution with a corner singularity and two temporal layers, refines the mesh at the corner and grades the time steps at both layers. 

\noindent
\section*{Acknowledgement.} 
The third author gratefully acknowledges the financial support from ANRF, India [Grant Number \\
MTR/2022/000130].

\bibliography{sandy}
\bibliographystyle{plain}
\end{document}